\documentclass{paper}
\usepackage{ngkarris-math}
\usepackage{subcaption}

\usepackage{authblk}

\graphicspath{{Figures/}}

\title{Using Linearized Optimal Transport to Predict \\ the Evolution of Stochastic Particle Systems}
\author[1]{Nicholas Karris}
\author[2]{Evangelos A.\ Nikitopoulos}
\author[3]{Ioannis Kevrekidis}
\author[4]{Seungjoon Lee}
\author[5]{Alexander Cloninger}
\affil[1,5]{University of California, San Diego}
\affil[2]{University of Michigan}
\affil[3]{Johns Hopkins University}
\affil[4]{California State University, Long Beach}

\date{November 16, 2025}

\begin{document}
\maketitle

\begin{abstract}
We develop an Euler-type method to predict the evolution of a time-dependent probability measure without explicitly learning an operator that governs its evolution.
We use linearized optimal transport theory to prove that the measure-valued analog of Euler's method is first-order accurate when the measure evolves ``smoothly.''
In applications of interest, however, the measure is an empirical distribution of a system of stochastic particles whose behavior is only accessible through an agent-based micro-scale simulation.
In such cases, this empirical measure does not evolve smoothly because the individual particles move chaotically on short time scales.
However, we can still perform our Euler-type method, and when the particles' collective distribution approximates a measure that \emph{does} evolve smoothly, we observe that the algorithm still accurately predicts this collective behavior over relatively large Euler steps, thus reducing the number of micro-scale steps required to step forward in time.
In this way, our algorithm provides a ``macro-scale timestepper'' that requires less micro-scale data to still maintain accuracy, which we demonstrate with three illustrative examples: a biological agent-based model, a model of a PDE, and a model of Langevin dynamics.
\end{abstract}

\tableofcontents

\section{Introduction}\label{sec:intro}

In this paper, we consider the problem of predicting the evolution of a time-dependent probability measure \(t \mapsto \mu_t\) on \(\R^d\), representing the distribution of particles in some physical system, given very limited information about the system's underlying dynamics.
Often, the physics governing the system can be approximated by an ``oracle operator'' \(E_h\) that takes as input the present state \(\mu_t\) of the system and returns an approximation \(E_h(\mu_t) \approx \mu_{t+h}\) after ``evolving'' according to the underlying dynamics for a short time \(h\).
As an example, \(E_h\) could be represented by an agent-based micro-scale model, in which individual particles act autonomously but exhibit some collective behavior when viewed as a group.

A natural way to understand such physical systems is to attempt to learn the oracle operator \(E_h\).
Many approaches to learning operators focus on tracking a sample of \(\mu_t\) through the actions of \(E_h\) and then interpolating the results \cite{lu2021learning, goswami2023physics, goswami2022deep, mhaskar2023local, kovachki2024data}.
However, in many applications, it is not possible to track the action of \(E_h\) on individual sample particles (e.g., due to randomness or statistical unreliability).
Moreover, these approaches typically attempt to learn the \emph{global} dynamics (both in time and space) of the operator, which in general requires data about the operator's actions on particles at many locations and many times.
Other approaches attempt to learn the dynamics by attempting to learn an SDE that approximates the observed evolution \cite{TLGD24,BR25}.
We consider a different perspective on tackling this problem.
Instead of learning the entire operator \(E_h\), we focus on the minimal data needed to predict the evolution of a particular initial distribution.
We ask whether it is possible to predict the measure's evolution \emph{without} needing to understand the underlying dynamics.
Explicitly, given some $\mu_t$ and $E_h(\mu_t) \approx \mu_{t+h}$ for only some small time step $h$, can we predict the measure $\mu_{t+H}$ for a larger time step $H$?
We should not need data about how \(E_h\) acts on other distributions at other times to do this.
If we can, then we can build an Euler-type method for approximating the long-term evolution of the distribution by making successive approximations, and this approximation method would only require \emph{local} knowledge of \(E_h\).
There are several potential approaches that attempt to only use local data, each with its own merits and drawbacks.
One drawback common to most approaches, though, is a reliance on computing and predicting summary statistics and then reconstructing a distribution with the predicted statistics.
We avoid this reliance on summary statistics and instead take a linearized optimal transport (LOT)--based approach informed by the geometry of the Wasserstein manifold that directly uses the measure-valued data to predict \(\mu_{t+H}\).

To explain the philosophy of our prediction algorithm, we briefly discuss curves on the Wasserstein manifold.
Precise definitions and results are reviewed in Section \ref{sec:prelim}.
Let $\cP_2(\R^d)$ be the space of Borel probability measures on $\R^d$ with finite second moment, and let $W_2$ be the ($2$-)Wasserstein distance on $\cP_2(\R^d)$.
If $\mu$ is a sufficiently ``smooth'' curve, written $t \mapsto \mu_t$, on the ``Wasserstein manifold'' $\W_2 \coloneqq (\cP_2(\R^d),W_2)$, then there exists a ``tangent field'' governing its evolution.
Explicitly, there is a time-dependent vector field $\bv_t \colon \R^d \to \R^d$ representing the time derivative of $\mu$ through the continuity equation
\[
\partial_t\mu_t + \nabla \cdot (\bv_t\mu_t) = 0.
\]
Intuitively speaking, the vector field $\bv_t$ describes the ``flow of mass'' of the evolving measure.
Furthermore, if $\mu$ always has a density, then its tangent field can be computed using optimal transport maps:
\[
    \bv_t = \lim_{h \to 0} \frac{T_{\mu_t}^{\mu_{t+h}} - \id}{h},
\]
where $T_{\mu_t}^{\mu_{t+h}}$ is the optimal transport map from $\mu_t$ to $\mu_{t+h}$ and $\id = \id_{\R^d}$.
This leads to the key insight that one can use optimal transport maps from \(\mu_t\) to \(\mu_{t+h}\) for small values of \(h\) to understand approximately where each infinitesimal piece of mass of \(\mu_t\) is moving.
Moreover, the interpretation of \(t \mapsto \bv_t\) as the tangent field to the curve \(\mu\) on the Wasserstein manifold suggests that the ``geodesic starting at \(\mu_t\) in the direction of \(\bv_t\),'' given by $s \mapsto (s\bv_t+\id)_{\#}\mu_t$, is a first-order accurate approximation of the true curve.
This result is made precise in Theorem \ref{thm:vectorfields}\ref{item:wp-first-order-Taylor}.
It also suggests a measure-valued analog of Euler's method for ODEs on the Wasserstein manifold, which we develop in Section \ref{subsec:wasserstein-euler}.

Euler's method requires complete knowledge of the dynamics of an ODE, and similarly, standard numerical methods for approximating measure evolution according to the continuity equation often require such knowledge \cite{EGH00,BS08,NFM18}.
As discussed in the first paragraph, our setting is fundamentally different because we suppose instead that we can only access dynamics information by querying an oracle $E_h$ for a short time $h$.
In this case, although we cannot know the true tangent field $t \mapsto \bv_t$, we can hope to approximate it by
\[
\bv_t^h \coloneqq \frac{T_{\mu_t}^{E_h(\mu_t)} - \id}{h} \approx \frac{T_{\mu_t}^{\mu_{t+h}} - \id}{h}.
\]
We then use this approximate tangent field to approximate $\mu_{t+H}$ as $\widetilde \mu_{t+H} \coloneqq (H\bv_t^h+\id)_{\#}\mu_t$, the geodesic starting at $\mu_t$ in the direction $\bv_t^h$ evaluated at time $H$.
Iterating such approximations leads to an Euler's method ``with incomplete information,'' which we discuss in Sections \ref{subsec:Euclideanpooreuler} and \ref{subsec:otvfs-finite-difs} for sufficiently ``smooth'' curves with densities at all times.
Note well that this Euler-type method allows us to predict the evolution of $\mu$ without attempting to learn anything about the operator~\(E_h\).

A key feature of the above-described approach is that Euler's method with incomplete information can be formulated in fairly arbitrary situations, e.g., for chaotically evolving discrete measures.
In this paper, we specifically consider cases arising from particle systems.
Let $x_1(t),\ldots,x_N(t)$ be a system of $N$ (random) particles in $\R^d$ evolving in time, and define $\mu_t^N \coloneqq N^{-1}\sum_{i=1}^N \delta_{x_i(t)}$.
If $\nu \coloneqq \mu_t^N$ and $\rho \coloneqq E_h(\mu_t^N) \approx \mu_{t+h}^N$, then the Euler-type method described above becomes the following algorithm:
\begin{enumerate}
    \item Compute a (discrete) optimal transport plan from \(\nu\) to \(\rho\), and compute the barycentric projection to obtain a map $T$.
    \item Use the map from Step 1 to form the difference quotient $\mathbf{v}_t^h \coloneqq \frac{T-\id}{h}$.
    \item For each point in the support of \(\nu\), take an Euler step using its current location and the velocity given by the ``optimal transport vector field'' $\mathbf{v}_t^h$.
    \item Set \(\widetilde\mu_{t+H}^N\) to be the uniform distribution on the points obtained by performing the Euler step from Step 3 on each point in the support of \(\nu\).
    \item Repeat Steps 1--4 using \(\widetilde\mu_{t+H}^N\) as the ``initial measure'' \(\nu\).
\end{enumerate}
Our actual algorithm, laid out in Section \ref{subsec:discrete-algorithm}, is more complicated.
Steps 1--5 above suffice for the present introductory discussion.

It is important to highlight that \textit{both} regularity properties demanded of $\mu$ in the earlier discussion of tangent fields and Euler-type methods are generally false for $\mu^N$, the curve $t \mapsto \mu_t^N$.
Certainly, $\mu^N$ never has a density because it is always a discrete measure.
Moreover, $\mu^N$ can fail to be ``smooth'' because the particles may evolve chaotically.
For example, if $x_1,\ldots,x_N$ are Brownian motions, then the sample paths are nowhere differentiable.
Therefore, \emph{a priori}, it is not clear the algorithm above actually works in this setting.
In many systems of interest, however, there is a ``smooth'' curve $t \mapsto \mu_t$ of measures with densities such that ``$\mu^N \approx \mu$'' for large $N$, e.g., if we start with the curve $\mu$ and the $x_i$'s are i.i.d.\ samples of $\mu$.
For certain such systems (Sections \ref{sec:chemotaxis}, \ref{sec:burgers}, and \ref{sec:halfmoon}), we experimentally demonstrate that the LOT-based algorithm described above can be used to predict the evolution of the particle distribution \(\mu^N\) without explicitly reconstructing the underlying distribution \(\mu\) or learning the underlying dynamics described by \(E_h\).

\subsection{Practical Utility for Agent-based Models}\label{subsec:motex}

One application for which our approximation algorithm is particularly useful is if \(E_h\) is given by a computationally expensive micro-scale particle simulation, as it is in Section \ref{sec:chemotaxis}.
Then we can think of \(\mu_t^N\) as the locations of the particles at time \(t\) and \(E_h(\mu_t^N)\approx \mu_{t+h}^N\) as their locations at the ``next time step'' in the simulation.
The benefit of accurately predicting \(\mu_{t+H}^N\) for a substantial time step \(H\gg h\) is that one can then skip potentially many micro-scale steps in the simulation and still understand how \(\mu^N\) evolves in time.
Indeed, after predicting \(\mu_{t+H}^N\), one can reinitialize the simulation with particles at those locations, take another small step in the simulation to obtain an estimate for \(\mu_{t+H+h}^N\), and use the prediction algorithm again to approximate \(\mu_{t+2H}^N\). 
Following strategies of some projective integration techniques \cite{GK03a, GK03b, SGOK05}, doing this successively gives an Euler-type method for approximating the system's behavior for a long time without the need to perform as many micro-scale steps.

Approximating distributional trajectories through an Euler-type method is not the only application of our proposed approach, however.
At its core, our approach gives a way to step forward in time without needing to compute the equivalent number of steps of the micro-scale model \(E_h\).
If computing \(E_h\) is expensive, then this means we get a comparatively cheap ``macro-scale timestepper,'' which has many applications beyond merely speeding up simulations.
For example, one could use such a macro-scale timestepper to perform system-level tasks such as fixed-point computations (including finding steady-state distributions which are dynamically unstable), bifurcation computations, stability analysis, and even coarse-grained feedback controller design and optimization \cite{SPK03,FVGSK25,GKT02}.
Thus, while we primarily evaluate our method on its accuracy in estimating distributional trajectories, the introduction of an accurate, cheap macro-scale timestepper has many implications beyond those directly discussed in the present paper.
As such, when evaluating the ``computational improvement'' of our method, we focus entirely on the number of micro-scale steps needed in its computation rather than the actual speed improvement of the simulation.
Even though the micro-scale models we use in our experiments are not \emph{extremely} expensive (and in Sections \ref{sec:burgers} and \ref{sec:halfmoon} are actually fairly cheap), one should have in mind applications where these micro-scale models are indeed the computational bottleneck as they are in, e.g., other agent-based biological models, climate models, etc.

One particular motivating class of such models are those of so-called ``fast-slow'' systems.
Such systems are difficult to analyze because each particle has multiple underlying variables impacting its behavior that update on vastly different time scales.
In general, the ``fast dynamics'' update on short time scales and generally manifest as chaotic short-time particle behavior.
For example, a particle may quickly change directions many times as a result of its interactions with nearby particles.
``Slow dynamics,'' on the other hand, update on much longer time scales and generally manifest as the eventual emergence of a collective behavior of the particles, e.g., the migration of bacteria toward a higher concentration of a food source.
Nominally, this means that in order to simulate such a system, a micro-scale simulation must be run with small time steps to capture the fast dynamics accurately, and \textit{many} such time steps are required for the slow dynamics to emerge.
Of course, this can be impractical when each step is computationally expensive.
One may hope to circumvent this by increasing the step size for the micro-scale simulation, thus requiring fewer steps to reach the desired end time, but some models computationally break down if the time step is too large (for example, the micro-scale model used in \ref{sec:chemotaxis} has this property --- see Remark \ref{rmk:chemotaxis-step-size-limit}) and others quickly lose accuracy \cite{GK03a,RGK03}.
Thus, a different approach is required to accurately simulate such systems.

Many approaches to understanding and more efficiently simulating these types of systems have been explored.
A few examples of so-called ``multi-scale methods'' are presented in \cite{KGHKRT03, KGH04, GKT02, Van03, ELV05, EE03}, and other ``projective integration methods'' are proposed in \cite{GK03a, GGK09} and others.
The techniques of Neural ODEs and Continuous Normalizing Flows can also be used to understand similar problems, e.g., in \cite{HMTFKWK22, THWDK20}.
Each of these techniques avoids the computational cost of simulating the system, but they all do so by restricting attention to specific classes of particle systems.
Another approach that works on general classes of particle systems is to learn an underlying SDE that models the observed behavior of the micro-scale simulator \cite{TLGD24,BR25}.
One could then hope to achieve a similar effect of skipping many expensive micro-scale steps by simulating the learned SDE (which is much cheaper) to the desired end time.
However, learning an SDE is prone to errors when only given local data about the behavior of the particles.
Moreover, these approaches assume that there is an underlying SDE that can accurately approximate the behavior of the micro-scale model at all.
In contrast, the method we propose makes efficient use of the limited local data, and it also assumes essentially nothing about the particular class of the underlying micro-scale simulator, only that the bulk behavior appears ``smooth.''
In particular, it does not require an assumption about the precise form of the dynamics of an agent-based model, nor does it attempt to learn those dynamics explicitly.
We avoid the computational cost not by making an assumption about the problem but by using optimal transport to detect the slow dynamics after only a small number of micro-scale steps.
The key insight is that, for systems that exhibit this multi-scale behavior, we can use an optimal transport map to effectively ignore the chaotic motion of \textit{individual} particles in favor of focusing on the slower collective behavior.

This insight has been used to a similar effect before: \cite{SGOK05} explores how a micro-scale model of bacterial chemotaxis in 1-D could be sped up by computing Euler-type steps after first sorting the locations of the bacteria.
In one dimension, sorting a discrete distribution is exactly the same as computing an optimal transport map, so it turns out that their approach is similar to ours.
This observation was the genesis of the present paper.
To demonstrate the practical utility of our algorithm, we perform a similar experiment with the same chemotaxis model (Section \ref{sec:chemotaxis}) as well as a model of a partial differential equation (Section \ref{sec:burgers}).
We show that our algorithm enables us to skip many micro-scale steps and still approximate the long-term behavior of the bacteria distribution, indicating that our Euler-type ``macro-scale timestepper'' is indeed reasonably accurate despite requiring fewer micro-scale steps.

Framing the algorithm in terms of optimal transport as opposed to particle sorting not only provides it theoretical motivation, as discussed earlier, but also enables its immediate generalization to particle systems in higher dimensions.
To demonstrate its efficacy on 2-D particle systems, we use the algorithm to simulate a model of overdamped Langevin dynamics in 2D (Section \ref{sec:halfmoon}).
While standard micro-scale models of Langevin dynamics (and other SDEs) are computationally cheap, the example is a good proof of concept for higher dimensions because it exhibits the same multi-scale behavior as described above.
Specifically, the particles move chaotically on short time scales, yet the overall distribution evolves slowly and predictably.

\subsection{Major Contributions}
\begin{itemize}
    \item We use linearized optimal transport to describe an analog of Euler's method for differential equations on the Wasserstein manifold (Section \ref{subsec:wasserstein-euler}).
    \item We prove that, under sufficient regularity conditions, our measure-valued Euler's method with step size \(H\) has local truncation error \(\cO(H^2)\) and corresponding global error \(\cO(H)\) (Lemma \ref{lem:lte-second-order} and Theorem \ref{thm:w2-euler-global-error}).
    \item We detail an algorithm that uses linearized optimal transport to approximate the evolution of discrete particle systems (Section \ref{sec:discrete}).
    \item We apply our approximation algorithm to three illustrative examples --- bacterial chemotaxis as described in \cite{SGOK05} (Section \ref{sec:chemotaxis}), a partial differential equation (Section \ref{sec:burgers}), and Langevin dynamics (Section \ref{sec:halfmoon}) --- and discuss its computational benefits.
    \item We demonstrate that our method outperforms equivalent methods (a particle-wise approach and JKOnet* \cite{TLGD24}) for approximating the distribution's evolution when only given access to local data.
\end{itemize}

\section{Preliminaries}\label{sec:prelim}

\subsection{Optimal Transport Background}

Let \(\cal P(\R^d)\) be the set of Borel probability measures on \(\R^d\), and let \(\cal P_2(\R^d)\) be the set of measures \(\mu\in\cal P(\R^d)\) such that \(\int_{\R^d}\|x\|_2^2\,d\mu(x) < \infty\).
For a measurable function \(T:\R^d\to \R^d\) and a measure \(\sigma\in\cal P(\R^d)\), the pushforward of \(\sigma\) by \(T\) is the measure \(T_\#\sigma\) defined by
\[
    (T_{\#}\sigma)(A) \defeq \sigma(T^{-1}(A))
\]
where \(A\subset\R^d\) is measurable and \(T^{-1}(A)\) is the preimage of \(A\) under \(T\).

Given \(\sigma, \mu\in\cP_2(\R^d)\), a classical problem (the ``Monge Problem'') is to seek an ``optimal'' map \(T\) transporting \(\sigma\) to \(\mu\) in the sense that \(T_\#\sigma = \mu\).
However, there does not always exist such a map, which leads to the more general ``Kantorovich formulation'':
\begin{defn}
    For \(\sigma, \mu \in\cP_2(\R^d)\), the (2-) Wasserstein distance between \(\sigma\) and \(\mu\) is
    \begin{equation}\label{eqn:wass-defn}
        W_2(\sigma, \mu) \defeq \min_{\gamma\in\Gamma_{\sigma,\mu}}\left(\int_{\R^d\times\R^d}\|x-y\|^2\,d\gamma(x,y)\right)^{\frac12},
    \end{equation}
    where \(\Gamma_{\sigma,\nu}\) is the set of all couplings between \(\sigma\) and \(\mu\), i.e., the set of all \(\gamma\in\cP_2(\R^d\times\R^d)\) such that
    \[
        \gamma(A\times \R^d) = \sigma(A), \qquad \gamma(\R^d\times A) = \mu(A)
    \]
    for all measurable \(A\subseteq \R^d\).

    The coupling \(\gamma\) that satisfies \eqref{eqn:wass-defn} is called the ``optimal coupling'' or the ``optimal plan.''
\end{defn}

It is well known that \(W_2\) defines a metric on \(\cP_2(\R^d)\) (see, e.g., \cite{PC19}).

We will only refer to ``optimal couplings'' a handful of times in this paper (and only to handle technical details).
Instead, we usually deal with optimal ``maps'' as in the Monge problem, and Brenier's Theorem is the standard result that guarantees the existence of such maps under more hypotheses:

\begin{thm}[Brenier \cite{Bre91}]\label{thm:brenier}
    Let \(\sigma,\mu\in\cal P_2(\R^d)\), and suppose \(\sigma\) has a density with respect to the Lebesgue measure \(\cL^d\).
    The optimal coupling \(\gamma\) that satisfies \eqref{eqn:wass-defn} is unique, and there exists a \(\sigma\)-a.e.\ unique map \(T:\R^d\to\R^d\) such that \(\gamma = (\id,T)_\#\sigma\).
    That is,
    \[
        W_2(\sigma,\mu) = \left(\int_{\R^d}\|x-T(x)\|^2\,d\mu(x)\right)^{\frac12}.
    \]
    Moreover, there exists a convex function \(\phi:\R^d\to \R\), $\sigma$-a.e.\ unique up to an additive constant, such that \(T = \nabla\phi\).
\end{thm}

\begin{defn}
    The map \(T\) given in Theorem \ref{thm:brenier} is the optimal transport map from \(\sigma\) to \(\mu\), and we denote it \(T_\sigma^\mu\).
\end{defn}

Optimal transport is also possible between discrete measures.
In general, there is not an optimal transport map between discrete measures, but the following result guarantees that one exists in the case of uniform measures on the same finite number of points.

\begin{prop}[Proposition 2.1 in \cite{PC19}]\label{prop:permutation}
    If \(\sigma = \frac1N\sum_{i=1}^N \delta_{x_i}\) and \(\mu = \frac1N\sum_{i=1}^N \delta_{y_i}\) are uniform measures on the same number of distinct points, then there exists an optimal transport map \(T_\sigma^\mu\) in the sense that
    \[
        W_2(\sigma,\mu) = \left(\frac1N\sum_{i=1}^N\|x_i - T_\sigma^\mu(x_i)\|^2\right)^{\frac12}
    \]
    Moreover, if \(T_\sigma^\mu\) is such a map, then there exists a permutation \(\tau\in S_N\) (the symmetric group on \(N\) elements) such that \(T_\sigma^\mu(x_i) = y_{\tau(i)}\) for all \(i = 1, \dots, N\). 
\end{prop}

Unlike in Theorem \ref{thm:brenier}, optimal transport maps as in Proposition \ref{prop:permutation} are not necessarily unique, even \(\sigma\)-a.e.
In this discrete case, the notation \(T_\sigma^\mu\) will denote a particular choice of an optimal map.

\subsection{Tangent Fields and Linearized Optimal Transport}\label{subsec:tangent-fields-and-LOT}

As Otto first observed in \cite{Ott01}, \(\cP_2(\R^d)\) has a formal infinite-dimensional Riemannian manifold structure.
Because of this, the language and philosophy of differential geometry are often used in optimal transport theory and its applications.
The key examples of interest to us and, more generally, the subject of linearized optimal transport (LOT) are the descriptions of the ``tangent space(s)'' and ``exponential map'' of the ``Wasserstein manifold.''
Below, we present the precise definition of the tangent space and a result that justifies its definition and underlies the philosophy and results of the present paper.

\begin{defn}[(8.0.2) in \cite{AGS08}]\label{defn:wp-tangent-space}
For $\sigma \in \cP_2(\R^d)$, define the \textit{tangent space} \(T_\sigma\W_2\) of the ``Wasserstein manifold'' \(\W_2 \defeq (\cP_2(\R^d),W_2)\) at \(\sigma\) to be the closure in $L^2(\R^d,\sigma;\R^d)$ of $\{\nabla \phi : \phi \in C_c^{\infty}(\R^d)\}$.
\end{defn}

\begin{thm}[Theorem 8.3.1, Proposition 8.4.5, and Proposition 8.4.6 in \cite{AGS08}] \label{thm:vectorfields}
    Let $I \subseteq \R$ be an open interval, let \(\mu: I \to \cP_2(\R^d)\) be an absolutely continuous curve in \(\W_2\) (written $t \mapsto \mu_t$), and let \(|\mu'|\) be the metric derivative of $\mu$.
    \begin{enumerate}[label=(\roman*),font=\normalfont,leftmargin=2em]
    \item There exists a Borel vector field \(I \times \R^d \ni (t,x)\mapsto \bv(t,x) = \bv_t(x) \in\R^d\) such that
    \begin{equation}\label{eqn:vfmin}
        \|\bv_t\|_{L^2(\mu_t)} \defeq \left(\int_{\R^d} \|\bv_t(x)\|_2^2\,d\mu_t(x)\right)^\frac12 \leq |\mu'|(t) \quad \text{ for } \cL^1 \text{-a.e.\ } t\in I
    \end{equation}
    and the continuity equation
    \begin{equation}\label{eqn:conteqn}
        \del_t\mu_t + \nabla\cdot(\bv_t\mu_t) = 0
    \end{equation}
    holds in the sense of distributions, i.e.,
    \begin{equation}\label{eqn:conteqndists}
        \int_I\int_{\R^d}\left[\del_t\phi(x,t) + \langle \bv_t(x), \nabla_x\phi(x,t)\rangle\right]d\mu_t(x)dt = 0 \quad \text{ for all } \phi\in C_c^{\infty}(\R^d\times I).
    \end{equation}
    The map $t \mapsto \mathbf{v}_t$ is $\cL^1$-a.e.\ uniquely determined by \eqref{eqn:vfmin} and \eqref{eqn:conteqn} and is called the {\rm tangent field} of $\mu$.\label{item:tangent-field-exist}
    \item If $\bv \colon I \times \R^d \to \R^d$ is any Borel vector field satisfying \eqref{eqn:conteqn}, then \eqref{eqn:vfmin} is satisfied if and only if \(\bv_t\in T_{\mu_t}\W_2\) for \(\cL^1\)-a.e.\ \(t\in I\).\label{item:tangent-field-tangent}
    \item If $\bv$ is the tangent field of $\mu$, then
    \begin{equation}\label{eqn:otprojs-firstorder}
        \lim_{h\to 0}\frac{W_2(\mu_{t+h}, (\id + h\bv_t)_{\#}\mu_t)}{|h|} = 0 \quad \text{ for } \cL^1 \text{-a.e.\ } t\in I.
    \end{equation}
    In particular, if \(\mu_t \ll \cL^d\) for all \(t\in I\), then for \(\cL^1\)-a.e.\ \(t\in I\),
    \begin{equation}\label{eqn:otderivs-firstorder}
        \lim_{h\to 0}\frac1h(T_{\mu_t}^{\mu_{t+h}} - \id) = \bv_t \quad \text{ in } L^2(\R^d, \mu_t; \R^d),
    \end{equation}
    where \(T_{\mu_t}^{\mu_{t+h}}\) is the optimal transport map from \(\mu_t\) to \(\mu_{t+h}\).\label{item:wp-first-order-Taylor}
    \end{enumerate}
\end{thm}

\begin{rmk}
    For a general choice of \(\sigma\in\cP_2(\R^d)\), the best one can hope for is that the map \(t\mapsto T_\sigma^{\mu_t}\in L^2(\R^d, \sigma; \R^d)\) is at most \(\tfrac12\)-H\"older continuous, even if the curve \(\mu\) is Lipschitz \cite{Gig11, MDC20}.
    However, \eqref{eqn:otderivs-firstorder} says that by fixing \(t\) and choosing \(\sigma = \mu_t\), the curve \(h\mapsto T_{\mu_t}^{\mu_{t+h}}\) becomes ``differentiable'' at \(h=0\).
\end{rmk}

The exponential map can now be computed by finding the tangent field of a geodesic between two points \(\mu_0, \mu_1\in\cP_2(\R^d)\).
One can show (Section 7.2 of \cite{AGS08}) that if $T_{\mu_0}^{\mu_1}$ is an optimal transport map from $\mu_0$ to $\mu_1$, then the curve
\[
    t\mapsto \mu_t \defeq ((1-t)\id + tT_{\mu_0}^{\mu_1})_\# \mu_0
\]
is a (constant-speed) geodesic in \(W_2\) from \(\mu_0\) to \(\mu_1\).
Furthermore, the vector field
\[
    \bv_t(x)\defeq (T_{\mu_0}^{\mu_1} - \id)\left(((1-t)\id + tT_{\mu_0}^{\mu_1})^{-1}(x)\right)
\]
is the tangent field of the geodesic \(\mu\) (Section 5.4 of \cite{San15}).
Therefore,  $\bv_0 = T_{\mu_0}^{\mu_1} - \id$ is the ``tangent vector'' to $\mu$ at $t=0$.
We conclude that the logarithm map, i.e., the inverse of the exponential map, at \(\mu_0\) is the map
\[
    \log_{\mu_0}^{\W_2} : \W_2 \to T_{\mu_0}\W_2, \qquad \log_{\mu_0}^{\W_2}(\mu_1) = \bv_0 = T_{\mu_0}^{\mu_1} - \id.
\]
Since $(\mathbf{v}_0 +\id)_{\#}\mu_0 = (T_{\mu_0}^{\mu_1})_{\#}\mu_0 = \mu_1$, it follows that the exponential map at $\mu_0$ is the map
\[
    \exp_{\mu_0}^{\W_2} \colon T_{\mu_0}\W_2 \to \W_2, \qquad \exp_{\mu_0}^{\W_2}(\mathbf{v}) = (\mathbf{v}+\id)_{\#}\mu_0.
\]
(Please see \cite{AG13, Lot08} for additional discussion.)
In particular, if $\mu$ is a general absolutely continuous curve with velocity field $\bv$, then Theorem \ref{thm:vectorfields}\ref{item:wp-first-order-Taylor} says
\[
    \mu_{t+h} \approx (\id+h\bv_t)_\#\mu_t = \exp_{\mu_t}^{\W_2}(h\bv_t)
\]
for small $h$, which makes sense from a differential geometric point of view.

An important consequence of the previous paragraph's description of the Wasserstein manifold's logarithm map is that \(\mu_1\mapsto T_{\mu_0}^{\mu_1}-\id\) ``linearizes'' \(\cP_2(\R^d)\);
hence the L in LOT.
Specifically, for a fixed reference measure \(\sigma\in\cP_2(\R^d)\), we can use the map $\mu\mapsto T_\sigma^\mu-\id$ to ``embed'' the space \(\cP_2(\R^d)\) into the linear space \(L^2(\R^d, \sigma; \R^d)\).
There are many advantages to this linear embedding; see, for example, \cite{WSBOR13,MC21} for computational speed-up and supervised learning and \cite{CHKM23} for performing PCA.

\section{Euler-type Methods on the Wasserstein Manifold}\label{sec:otvfs}

\subsection{Motivation from Euclidean Space}\label{subsec:Euclideanpooreuler}

In the sections following this one, we develop Euler-type methods for approximating the solutions to (nice) ``ODEs'' in the Wasserstein manifold.
To understand these methods, it is helpful to discuss their analogs in the context of ODEs in \(\R^d\).
These analogs in \(\R^d\) themselves are not useful methods for numerically solving ODEs, but they are useful to present to build intuition for what our Euler-type method is doing geometrically.
To this end, suppose we have a particle whose location at time \(t\) is described by some curve \(\gamma(t)\) solving the initial value problem
\[
    \gamma(0) = x_0, \qquad \gamma'(t) = f(t,\gamma(t)).
\]
Our goal is to estimate \(\gamma(T)\) for some large \(T\).
If we happen to have access to the underlying ``dynamics'' function \(f\), then we can use ordinary Euler's method:
For a time step \(H>0\) and discrete times \(t_n = nH\), approximate \(\gamma(t_n) \approx \gamma_{(n)}\), where \(\gamma_{(n)}\) is defined recursively by
\[
    \gamma_{(0)} = x_0, \qquad \gamma_{(n+1)} = \gamma_{(n)} + Hf(t_n, \gamma_{(n)}) = \exp_{\gamma_{(n)}}^{\R^d}(Hf(t_n,\gamma_{(n)})),
\]
where $\exp_{\gamma_{(n)}}^{\R^d}$ denotes the exponential map of $\R^d$ at $\gamma_{(n)}$.
Under sufficient regularity conditions on \(f\), it is well known that Euler's method is first-order accurate.
The measure-valued analog of this method is developed in Section \ref{subsec:wasserstein-euler}, and we prove an analogous accuracy result.

Suppose now that we have ``imperfect information'' about our system's dynamics, i.e., that we do not have explicit access to the underlying function \(f\).
Then we cannot perform Euler's method exactly.
Instead, suppose we have access to some oracle \(E_h\) that, given a time \(t\) and a particle location \(\gamma\), returns an approximation of where that particle would be at time \(t+h\) for some small time step \(h>0\).
(For us, this oracle will be a micro-scale numerical simulation that is only accurate for very small time steps \(h\) and is computationally expensive, so using it to simulate long-term behavior is impractical.)
Given this information, we attempt to approximate \(f(t,\gamma)\) by
\[
    f(t,\gamma) \approx \frac{E_h(t,\gamma) - \gamma}{h} \eqdef \hat f(t,\gamma).
\]
Using this approximation, we then define the Euler-type method
\[
    \gamma_{(0)} = x_0, \qquad \gamma_{(n+1)} = \gamma_{(n)} + H\hat f(t_n,\gamma_{(n)}) = \exp_{\gamma_{(n)}}^{\R^d}\big(H\hat f(t_n,\gamma_{(n)})\big).
\]
Of course, this method also makes sense for other approximations \(\hat f\) of \(f\), and the error of the method depends on how well \(\hat f\) approximates \(f\).
While precise statements can be made, we omit the details for brevity.
The measure-valued analog of this Euler-type method with a finite-difference approximation of the ``dynamics'' function is discussed in Section \ref{subsec:otvfs-finite-difs}.
In Section \ref{sec:discrete}, we describe the ``discrete case'' of this approximation algorithm (in which \(E_h\) is a given micro-scale simulator).

\subsection{Euler's Method on the Wasserstein Manifold}\label{subsec:wasserstein-euler}

According to Theorem \ref{thm:vectorfields}'s characterization of tangent fields to curves on $\W_2$, the measure-valued analog of the ODE $\gamma'(t) = f(t,\gamma(t))$ is the continuity equation
\begin{equation}\label{eqn:wp-ode}
    \partial_t\mu_t + \nabla \cdot(F(t,\mu_t)\,\mu_t) = 0,
\end{equation}
where \(F\) is some ``(time-dependent) vector field on the Wasserstein manifold.''
Given such a ``vector field'' \(F\) and an initial measure \(\mu_0\), the measure-valued analog of Euler's method is the following:
For discrete times \(t_n = nH\), set
\begin{equation}\label{eqn:wp-euler-method}
    \mu_{(0)} \defeq \mu_{0}, \qquad \mu_{(n+1)} \defeq \exp_{\mu_{(n)}}^{\W_2}(HF(t_n,\mu_{(n)})) = (\id + HF(t_n,\mu_{(n)}))_\#\mu_{(n)}.
\end{equation}
With this definition, one should expect that \(\mu_{(n)} \approx \mu_{t_n}\) is first-order accurate under appropriate assumptions on \(F\).
The following result makes this precise.

\begin{thm}\label{thm:w2-euler-global-error}
    Let \(F:[0,T]\times \cP_2(\R^d) \to (\R^d)^{\R^d} = \{\textnormal{functions }\R^d\to\R^d\}\) be a map such that \(F(t,\mu) \in T_\mu \W_2\) for all \(t\in[0,T]\) and \(\mu\in\cP_2(\R^d)\).
    Suppose \(F\) satisfies the following Lipschitz conditions:
    There exist \(L_1,L_2\geq 0\) such that for all \(t\in[0,T]\), \(x_1,x_2\in\R^d\), and \(\mu_1,\mu_2\in\cP_2(\R^d)\),
    \[
        \|F(t,\mu)(x_1) - F(t,\mu)(x_2)\|_2 \leq L_1\|x_1-x_2\|_2 \quad \text{ and }
    \]
    \[
        \|F(t,\mu_1) - F(t,\mu_2)\|_{L^2(\mu_i)} \leq L_2W_2(\mu_1,\mu_2) \quad (i=1,2).
    \]
    If \(\mu:[0,T]\to \cP_2(\R^d)\) is an absolutely continuous curve solving the continuity equation \eqref{eqn:wp-ode} on $(0,T)$ in the sense of distributions and the function \((t,x)\mapsto \bv_t(x) \defeq F(t,\mu_t)(x)\) is (jointly) \(C^1\), then for \(\mu_0\)-a.e.\ \(x\in\R^d\), the characteristic ODE
    \[
        \gamma_x(0) = x, \qquad \frac{d}{dt}\gamma_x(t) = \bv_t(\gamma_x(t))
    \]
    admits a unique \(C^2\) solution.
    Furthermore, if
    \[
        M \defeq \left(\int_{\R^d}\max_{r\in[0,T]}\|\gamma_x''(r)\|_2^2\,d\mu_0(x)\right)^\frac12 < \infty,
    \]
    then for a time step \(H>0\) and discrete times \(t_n \defeq nH\), the sequence of Euler approximations iteratively defined by
    \[
        \mu_{(0)} \defeq \mu_0, \qquad \mu_{(n+1)} \defeq (\id + HF(t_n,\mu_{(n)}))_\#\mu_{(n)}
    \]
    satisfies
    \[
        W_2(\mu_{(n)}, \mu_{t_n}) \leq \frac{HM}{2(L_1+L_2)}\left(e^{nH(L_1+L_2)}-1\right)
    \]
    for all $n$ such that $t_n = nH \leq T$.
\end{thm}

The rest of this section is devoted to proving this result.
To begin, observe that \eqref{eqn:otprojs-firstorder} says the local truncation error of Euler's method on the Wasserstein manifold is \(o(H)\).
We now show that under our regularity assumptions on \(F\), this method has local truncation error \(\cO(H^2)\).
This is the same as the local truncation error of Euler's method for curves in Euclidean space with bounded second (e.g., see (6.2.5) of \cite{Atk89}). 

\begin{lem}[Lemma 8.1.4 and Proposition 8.1.8 in \cite{AGS08}] \label{lem:cty-soln-rep}
    Suppose \(\mu:[0,T]\to \cP_2(\R^d)\) is an absolutely continuous curve satisfying \eqref{eqn:conteqn} on $(0,T)$ with velocity field \(\bv\) satisfying
    \begin{equation}\label{eqn:8dot1dot2}
        \int_0^T\int_{\R^d}\|\bv_t(x)\|_2\,d\mu_t(x)dt < \infty
    \end{equation}
    and
    \begin{equation}\label{eqn:8dot1dot8}
        \int_0^T\left(\sup_{x\in B}\|\bv_t(x)\|_2 + \operatorname{Lip}(\bv_t,B)\right)\,dt < \infty
    \end{equation}
    for all compact set \(B\subset \R^d\).
    Then for \(\mu_0\)-a.e.\ \(x\in\R^d\), the ODE
    \begin{equation}\label{eqn:charode}
        \gamma_x(0) = x, \qquad \frac{d}{dt}\gamma_x(t) = \bv_t(\gamma_x(t))
    \end{equation}
    admits a unique solution defined on \([0,T]\), and
    \begin{equation}
        \mu_t = (T_t)_\#\mu_0 \qquad\text{ for all } t\in[0,T],
    \end{equation}
    where \(T_t(x) \defeq \gamma_x(t)\).
\end{lem}

\begin{lem}\label{lem:lte-second-order}
    Assume the hypotheses of Lemma \ref{lem:cty-soln-rep}.
    Fix \(t\in[0,T]\), let \(H>0\) be such that \(t+H \leq T\), and let \(\gamma_x:[t,t+H]\to\R^d\) be the solution to
    \[
        \gamma_x(t) = x, \qquad \gamma_x'(s) = \frac{d}{ds}\gamma_x(s) = \bv_s(\gamma_x(s))
    \]
    If \(\gamma_x\) is \(C^2\) for \(\mu_t\)-a.e.\ \(x\in\R^d\) and
    \[
        \int_{\R^d} \max_{r\in[t,t+H]}\|\gamma_x''(r)\|_2^2\,d\mu_t(x)<\infty,
    \]
    then
    \[
        W_2(\mu_{t+H}, (\id+H\bv_t)_\#\mu_t) \leq \frac{H^2}{2}\left(\int_{\R^d}\max_{r\in[t,t+H]}\|\gamma_x''(r)\|_2^2\,d\mu_t(x)\right)^\frac12.
    \]
\end{lem}

\begin{proof}
    Define \(T_s:\R^d\to\R^d\) by \(T_s(x) = \gamma_x(s)\).
    By Lemma \ref{lem:cty-soln-rep}, $\mu_{t+H} = (T_{t+H})_\#\mu_t$, so
    \begin{align*}
        W_2^2(\mu_{t+H}, (\id+H\bv_t)_\#\mu_t) &\leq \int_{\R^d} \|T_{t+H}(x) - (\id+H\bv_t)(x)\|_2^2\,d\mu_t(x) \\
        &= \int_{\R^d}\|\gamma_x(t+H) - [x+H\bv_t(x)]\|_2^2\,d\mu_t(x).
    \end{align*}
    If \(\gamma_x\) is \(C^2\) for a particular \(x\in\R^d\), then
    \begin{align*}
        \gamma_x(t+H) &= \gamma_x(t) + H\gamma_x'(t) + \int_t^{t+H}\int_t^s \gamma''(r)\,drds \\
        &= x + H\bv_t(x) + \int_t^{t+H}\int_t^s \gamma''(r)\,drds.
    \end{align*}
    Hence, if \(\gamma_x\) is \(C^2\) for \(\mu_t\)-a.e.\ \(x\in\R^d\), then
    \begin{align*}
        W_2^2(\mu_{t+H}, (\id+H\bv_t)_\# \mu_t) &\leq  \int_{\R^d}\left\|\int_t^{t+H}\int_t^s\gamma_x''(r)\,drds\right\|_2^2d\mu_t(x) \\
        &\leq \int_{\R^d}\frac{H^4}{4}\max_{r\in[t,t+H]}\|\gamma_x''(r)\|_2^2\,d\mu_t(x).\qedhere
    \end{align*}
\end{proof}

As in the Euclidean case, the fact that the local truncation error of Euler's method on the Wasserstein manifold is \(\cO(H^2)\) implies that, under our Lipschitz conditions on \(F\), the global error is first order.
This is exactly the content of Theorem \ref{thm:w2-euler-global-error}.

\begin{proof}[Proof of Theorem \ref{thm:w2-euler-global-error}]
    We first prove that the ODE \eqref{eqn:charode} admits a unique solution on \([0,T]\) for \(\mu_0\)-a.e.\ \(x\in\R^d\).
    By Lemma \ref{lem:cty-soln-rep}, it suffices to show that \(\bv_t\defeq F(t,\mu_t)\) satisfies \eqref{eqn:8dot1dot2} and \eqref{eqn:8dot1dot8}.
    For the first, note that
    \begin{align*}
        \int_0^T\int_{\R^d}\|\bv_t(x)\|_2\,d\mu_t(x)dt &\leq \int_0^T\int_{\R^d}\left[\|\bv_t(0)\|_2 + \|\bv_t(x) - \bv_t(0)\|_2\right]d\mu_t(x)dt \\
        &\leq \int_0^T\int_{\R^d}\left[\|\bv_t(0)\|_2 + L_1\|x\|_2\right]d\mu_t(x)dt \\
        &= \int_0^T\|\bv_t(0)\|_2\,dt + L_1\int_0^T\int_{\R^d}\|x\|_2\,d\mu_t(x)dt.
    \end{align*}
    We assume \((t,x)\mapsto \bv_t(x)\) is \(C^1\), so
    \[
        \int_0^T\|\bv_t(0)\|_2\,dt < \infty,
    \]
    and \(t\mapsto \int_{\R^d}\|x\|_2\,d\mu_t(x)\) is continuous (because \(t\mapsto \mu_t\) is continuous with respect to \(W_2\) by hypothesis and \(\mu\mapsto \int_{\R^d}\|x\|_2\,d\mu(x)\) is continuous with respect to \(W_2\) because \(\|x\|\) grows less than quadratically (see Remark 7.1.11 in \cite{AGS08})), so
    \[
        \int_0^T\int_{\R^d}\|x\|_2\,d\mu_t(x)dt < \infty.
    \]
    Hence, we have \eqref{eqn:8dot1dot2}.
    For \eqref{eqn:8dot1dot8}, if \(B\subset \R^d\) is compact, then
    \begin{align*}
        \int_0^T\left(\sup_{x\in B}\|\bv_t(x)\|_2 + \operatorname{Lip}(\bv_t,B)\right)dt &\leq T\left(\sup_{t\in[0,T],x\in B}\|\bv_t(x)\|_2 + L_1\right) < \infty
    \end{align*}
    again because \((t,x)\mapsto \bv_t(x)\) is \(C^1\).
    By Lemma \ref{lem:cty-soln-rep}, the ODE \eqref{eqn:charode} has a unique globally defined solution \(\gamma_x\) for \(\mu_0\)-a.e.\ \(x\in\R^d\), and since \((x,t)\mapsto \bv_t(x)\) is \(C^1\), the map \(t\mapsto \gamma_x(t)\) is \(C^2\) for such \(x\) (since its second derivative is \(\gamma_x''(t) = \frac{d}{dt}\bv_t(\gamma_x(t))\), which is continuous by \(C^1\) of \(\bv_t(x)\)).
    
    Now, for
    \[
        M \defeq \left(\int_{\R^d}\max_{r\in[0,T]}\|\gamma_x''(r)\|_2^2\,d\mu_0(x)\right)^\frac12 < \infty,
    \]
    we will show
    \begin{equation}\label{eqn:w2-euler-global-iterative}
        W_2(\mu_{(n+1)}, \mu_{t_{n+1}}) \leq (1+H(L_1+L_2))W_2(\mu_{(n)},\mu_{t_n}) + \frac{H^2}{2}M.
    \end{equation}
    By the triangle inequality for \(W_2\), we have
    \begin{align*}
        W_2(\mu_{(n+1)},\mu_{t_{n+1}}) &= W_2((\id+HF(t_n, \mu_{(n)}))_\#\mu_{(n)}, \mu_{t_{n+1}}) \\
        &\leq W_2((\id+HF(t_n, \mu_{(n)}))_\#\mu_{(n)}, (\id+HF(t_n, \mu_{t_n}))_\#\mu_{(n)}) \tag{\(*\)} \\
        &+ W_2((\id+HF(t_n, \mu_{t_n}))_\#\mu_{(n)}, (\id+HF(t_n, \mu_{t_n}))_\#\mu_{t_n}) \tag{\(\dag\)} \\
        &+ W_2((\id+HF(t_n, \mu_{t_n}))_\#\mu_{t_n}, \mu_{t_{n+1}}). \tag{\(\ddag\)} \\
    \end{align*}
    We will handle each term separately.
    For \((*)\), note that for any \(\mu\in\cP_2(\R^d)\) and any maps \(T,S:\R^d\to\R^d\), the measure \(\gamma = (T,S)_\#\mu\) is a (generally suboptimal) coupling of \(T_\#\mu\) and \(S_\#\mu\), and so
    \[
        W_2(T_\#\mu, S_\#\mu) \leq \left(\int_{\R^d}\|T(x) - S(x)\|^2\,d\mu(x)\right)^{\frac12}.
    \]
    Thus, we have
    \begin{align*}
        (*) &\leq \left(\int_{\R^d}\left\|(\id+HF(t_n, \mu_{(n)}))(x) - (\id+HF(t_n, \mu_{t_n}))(x)\right\|^2\,d\mu_{(n)}(x)\right)^{\frac12} \\
        &= \|(\id+HF(t_n, \mu_{(n)})) - (\id+HF(t_n, \mu_{t_n}))\|_{L^2(\mu_{(n)})} \\
        &= H\|F(t_n,\mu_{(n)}) - F(t_n, \mu_{t_n})\|_{L^2(\mu_{(n)})} \\
        &\leq HL_2W_2(\mu_{(n)},\mu_{t_n}).
    \end{align*}
    
    For \((\dagger)\), note that for any Lipschitz map \(T:\R^d\to\R^d\) and any measures \(\sigma,\mu\in\cP_2(\R^d)\) with optimal coupling \(\gamma\), we have that \((T\times T)_\#\gamma\) is a coupling of \(T_\#\sigma\) and \(T_\#\mu\), and so
    \begin{align*}
        W_2(T_\#\sigma, T_\#\mu) &\leq \left(\int_{\R^d\times \R^d}\|T(x) - T(y)\|^2\,d\gamma(x,y)\right)^{\frac12} \\
        &\leq \|T\|_{\operatorname{Lip}}\left(\int_{\R^d\times\R^d}\|x-y\|\,d\gamma(x,y)\right)^{\frac12} \\
        &= \|T\|_{\operatorname{Lip}}W_2(\sigma,\mu).
    \end{align*}
    Hence
    \begin{align*}
        (\dag) &\leq \|\id + HF(t_n, \mu_{t_n})\|_{\operatorname{Lip}}W_2(\mu_{(n)}, \mu_{t_n}) \\
        &\leq (1 + HL_1)W_2(\mu_{(n)}, \mu_{t_n}).
    \end{align*}
    
    Finally, for \((\ddag)\), we have
    \begin{align*}
        (\ddag) &\leq \frac{H^2}{2}M
    \end{align*}
    by Lemma \ref{lem:lte-second-order}.
    
    Now, we recall the following elementary fact: If $\beta \geq 0$, $\alpha > 0$, and \((a_n)_{n \in \N}\) is a sequence of nonnegative numbers satisfying the recursive property that \(a_{n+1} \leq (1+\alpha)a_n + \beta\), then
    \[
        a_n \leq e^{n\alpha}\left(a_0 + \frac\beta\alpha\right) - \frac\beta\alpha.
    \]
    Applying this to (\ref{eqn:w2-euler-global-iterative}) with \(\alpha = H(L_1 + L_2)\) and \(\beta = \frac{H^2}{2}M\) gives
    \[
        W_2(\mu_{(n)}, \mu_{t_n}) \leq \frac{HM}{2(L_1+L_2)}\left(e^{nH(L_1+L_2)}-1\right),
    \]
    as desired.
\end{proof}

\begin{rmk}
    Before moving on, it is important to discuss the practical utility of Theorem \ref{thm:w2-euler-global-error} and its similarity to the standard error bound result for ordinary Euler's method in \(\R^d\) (e.g., see (6.2.13) of \cite{Atk89}).
    We note that the role of the Lipschitz constant \(L_2\) is exactly analogous to the role of the Lipschitz constant in the proof of the \(\R^d\) Euler result, and the role of \(L_1\) is to ensure that the vector fields returned by \(F\) are sufficiently regular that the curves \(\gamma_x\) themselves can be well approximated using Euler's method.
    This is also the role of the assumption that \((t,x)\mapsto F(t,\mu_t)(x)\) is jointly \(C^1\) --- this guarantees that the curves \(\gamma_x\) have bounded second derivative, and so the definition of \(M\) is sensible.
    These hypotheses may seem rather restrictive, but there are indeed ``natural'' examples which satisfy them.
    
    As an example, consider the standard SDE \(dX_t = dW_t\) with initial condition \(X_0 \sim \mu_0 = \delta_0\), for which the distribution associated with \(X_t\) is known to be \(\mu_t = \cN(0,t)\).
    Since optimal transport maps between Gaussian distributions has a closed form, we can use \eqref{eqn:otderivs-firstorder} to compute the minimal-energy velocity flow field \(\bv_t\) in the continuity equation:
    \begin{equation}\label{eqn:gaussian-true-vf}
        \bv_t(x) = \lim_{h\to 0}\frac{T_{\mu_t}^{\mu_{t+h}}(x) - x}{h} = \lim_{h\to 0}\frac{\sqrt{\frac{t+h}{t}}-1}{h}x = \frac{x}{2t}.
    \end{equation}
    Thus, we can recover pure Gaussian diffusion from the setup in Theorem \ref{thm:w2-euler-global-error} by setting \(F(t,\mu)(x) = \frac{x}{2t}\) for \(t>0\) (any \(F\) which satisfies \(F(t,\cN(0,t)) = \frac{x}{2t}\) will yield the correct evolution).
    Note that for \(t>0\), \(F\) is jointly \(C^1\) and has Lipschitz constants \(L_1 = \frac{1}{t}\) and \(L_2 = 0\).

    However, for \(t=0\), we must be more careful.
    For the initial distribution \(\mu_0 = \delta_0\), no velocity flow field exists which will ``separate'' the mass and create a continuous distribution --- any velocity flow field will cause the point-mass to flow as a single point, rather than evolve as a distribution with density.
    In this case, Theorem \ref{thm:w2-euler-global-error} still holds, but it reduces to the standard bounds on Euler's method for an ODE on a single particle.
    To exactly mimic the evolution \(\mu_t = \cN(0,t)\), we must start with some initial distribution \(\mu_{t_0} = \cN(0,t_0)\) for \(t_0>0\), in which case the hypotheses are indeed satisfied for all \(t>t_0\).
    
    The example of standard Gaussian diffusion is perhaps the simplest, and so one might wish to check that the hypotheses are also satisfied in some more interesting examples.
    Unfortunately, even for well understood systems, it is often difficult to obtain an explicit expression for the minimal-energy velocity flow field \(\bv_t\) that satisfies the continuity equation
    \[
        \del_t\mu_t + \nabla\cdot(\bv_t\mu_t) = 0,
    \]
    and so it is difficult to write down a function \(F\) that yields the evolution \(\mu_t\).

    For example, consider an SDE of the form \(dX_t = v(X_t)dt + dW_t\).
    Even though it is well known that the distribution \(\mu_t\) associated with \(X_t\) evolves according to the Fokker-Planck equation
    \[
        \del_t\mu_t + \nabla\cdot(v\mu_t) - \frac12\Delta\mu_t = 0,
    \]
    the drift velocity field \(v\) in the SDE is in general not the same as the velocity flow field \(\bv_t\) in the continuity equation (e.g., see the above example where \(v = 0\) but \(\bv_t = \frac{x}{2t}\). 
    This means that \(\bv_t\) actually does have to be computed via optimal transport maps as in \eqref{eqn:otderivs-firstorder}, and since optimal transport maps between arbitrary continuous distributions do not have closed-form expressions in general, neither will the velocity flow field \(\bv_t\).
    (It is true that one can use the Fokker-Planck equation to find \emph{some} velocity field which solves the continuity equation with \(\mu_t\), but finding the minimal-energy one --- and hence the appropriate value of \(F\) --- requires solving an optimization problem that is essentially equivalent to solving for the optimal transport map.)
    Thus, even when \(\mu_t\) has a nice expression (as can be the case for SDEs), the velocity flow field \(\bv_t\) may not, and so the hypotheses are difficult to explicitly check in practice.

    Thus, we wish to be clear that our goal in presenting these results is not so that we can directly apply them to ``real world'' examples.
    Instead, we wish to provide theoretical motivation for the discrete algorithm we present in Section \ref{sec:discrete}, to elucidate the analogy to ordinary Euler's method, and to demonstrate that similar guarantees are, at least in theory, possible to obtain.
    In this way, the primary takeaway of the paper is not that our algorithm can be \emph{guaranteed} to work on a particular system \emph{a priori}, but rather that it empirically performs well on example systems for which it is believable that the appropriate hypotheses hold, even if one cannot check them explicitly.
\end{rmk}

\subsection{Euler's Method with Imperfect Information}\label{subsec:otvfs-finite-difs}

The above discussion shows that, under sufficient regularity assumptions, Euler's method on the Wasserstein manifold can approximate solutions to the continuity equation with tangent field prescribed by a ``vector field'' \(F\).
However, as discussed in Section \ref{subsec:Euclideanpooreuler}, we rarely have access to \(F\) in practice, so we cannot explicitly compute the velocity field to be used in Euler's method.
Instead, we will only have access to some oracle \(E_h\) that, given a time \(t\) and a measure \(\mu\), will return an approximation of what that measure would be at time \(t+h\) if allowed to evolve according to \(F\) for some small time step \(h\).
That is, if \(\mu\) is a solution to the continuity equation with \(\bv_t = F(t,\mu_t)\), then the oracle is such that
\[
    E_h(t,\mu_t) \approx \mu_{t+h}.
\]
To perform an Euler-type method, then, we must first approximate the appropriate velocity field, and \eqref{eqn:otderivs-firstorder} hints that we should approximate it with
\[
     \bv_t^h(x) \defeq \frac{T_{\mu_t}^{E_h(t,\mu_t)}(x)-x}{h} \approx \frac{T_{\mu_t}^{\mu_{t+h}}(x) - x}{h}.
\]
Unfortunately, even if \(E_h(\mu_t) = \mu_{t+h}\) exactly, the \(L^2\) convergence in \eqref{eqn:otderivs-firstorder} is not enough to guarantee that \(\bv_t^h(x)\to \bv_t(x)\) pointwise as \(h\to 0\).
However, the \(L^2\) convergence is still good enough to give a reasonable bound on the local truncation error of using this finite difference approximation in Euler's method.

\begin{cor}\label{cor:ptwise}
    Suppose \(\mu_t \ll \cL^d\) and \(h>0\). Set
    \begin{equation*}
        \bv_t^h (x) \defeq \frac{T_{\mu_t}^{E_h(t,\mu_t)}(x) - x}{h}.
    \end{equation*}
    Then
    \begin{equation}\label{eqn:ptwiseapproxs}
        W_2(\mu_{t+H}, (\id + H\bv_t^h)_{\#}\mu_t) \leq H\|\bv_t - \bv_t^h\|_{L^2(\mu_t)} + o(H)
    \end{equation}
    as \(H\to 0\), and if \(E_h(t,\mu_t) = \mu_{t+h}\), then \(\|\bv_t - \bv_t^h\|_{L^2(\mu_t)}\to 0\) as \(h\to 0\).
\end{cor}

\begin{proof}
    First note that, since we assume \(\mu_t\ll \cL^d\), we can apply \eqref{eqn:otderivs-firstorder} to give that \(\|\bv_t-\bv_t^h\|_{L^2(\mu_t)}\to 0\) as \(h\to 0\) when \(E_h(t,\mu_t) = \mu_{t+h}\).
    
    For the inequality, by \eqref{eqn:otprojs-firstorder} we have
    \begin{align*}
        W_2(\mu_{t+H}, (\id + H\bv_t^h)_{\#}\mu_t) &\leq W_2(\mu_{t+H}, (\id + H\bv_t)_{\#}\mu_t) + W_2((\id + H\bv_t)_{\#}\mu_t, (\id + H\bv_t^h)_{\#}\mu_t) \\
        &\leq o(H) + \|(\id + H\bv_t) - (\id + H\bv_t^h)\|_{L^2(\mu_t)} \\
        &= o(H) + H\|\bv_t - \bv_t^h\|_{L^2(\mu_t)}
    \end{align*}
    where in the second line we once again use the fact that for any maps \(T,S:\R^d\to\R^d\), we have \(W_2(T_\#\mu, S_\#\mu) \leq \|T-S\|_{L^2(\mu)}\) because \((T,S)_\#\mu\) gives a (generally suboptimal) coupling between \(T_\#\mu\) and \(S_\#\mu\).
\end{proof}

This is the local truncation error of the Euler-type method
\begin{equation}\label{eqn:wp-euler-approx-fd}
    \mu_{(0)} = \mu_{0}, \qquad \mu_{(n+1)} = (\id + H\bv_t^h)_{\#}\mu_{(n)}.
\end{equation}
As was alluded to in Section \ref{subsec:Euclideanpooreuler}, the local truncation error of this method depends on the accuracy of the approximation of the ``tangent vector'' \(\bv_t = F(t,\mu_t)\).

\section{Euler-type Method for Empirical Measures of Particle Systems}\label{sec:discrete}

The results in Section \ref{sec:otvfs} show that an Euler-type scheme can effectively predict the evolution of a measure-valued curve (under certain conditions) even with incomplete knowledge of the underlying ``dynamics'' governing the curve's evolution.
In practical applications, one is often interested in predicting the evolution of (random) systems of particles \(\{x_i(t)\}_{i=1}^N\) in \(\R^d\).
In such cases, the measure-valued curve one seeks to predict is the empirical measure
\[
    t \mapsto \mu_t^N \defeq \frac1N\sum_{i=1}^N\delta_{x_i(t)},
\]
and the oracle-type information one has is a particle-wise micro-scale simulation.
Using this particle-wise simulation, we can construct an ``oracle operator'' \(E_h\) on discrete measures.
If $\nu = \frac{1}{N}\sum_{i=1}^N \delta_{y_i}$ is a discrete measure, then $E_h(\nu) \coloneqq \frac{1}{N}\sum_{i=1}^N \delta_{z_i}$, where $\{z_i\}_{i=1}^N$ is the output of running the micro-scale simulation for time \(h\) initialized with particles at the locations $\{y_i\}_{i=1}^N$.
With this oracle in hand, we may hope to apply the Euler-type method described in Section \ref{subsec:otvfs-finite-difs} to $\mu^N$.
Specifically, if
\begin{align}
    \bv_t^{h,N}(x_i(t)) & \defeq \frac{T_{\mu_t^N}^{E_h(\mu_t^N)}(x_i(t)) - x_i(t)}{h} \; \qquad\qquad i=1,\ldots,N, \label{eqn:discrete-findif-vf} \\
     \widetilde{\mu}_{t+H}^N & \defeq  \left(\id + H\bv_t^{h,N}\right)_{\#}\mu_t^N = \frac1N\sum_{i=1}^N \delta_{x_i(t) + H\bv_t^{h,N}(x_i(t))},\label{eqn:wp-findif-euler-onestep}
\end{align}
then we may hope that $\widetilde{\mu}_{t+H}^N \approx \mu_{t+H}^N$.
Unfortunately, none of the theory from Section \ref{sec:otvfs} applies in this setting:
$\mu_t^N$ clearly does not have a density, and in many systems of interest, the particles' evolution is too rough in time for the curve $\mu^N$ to be absolutely continuous.
(Suppose $x_i$ is, e.g., a Brownian motion.)
In particular, $\mu^N$ has no tangent field, so there is no reason to believe that $\bv^{h,N}$ can tell us anything about the true value $\mu_{t+H}^N$.
The key idea of our examples is to sidestep this major issue by assuming that $\mu^N$ is an empirical estimate of a ``nice'' curve $\mu$ with velocity field $\bv$.
In this case, $\mu^N \approx \mu$ for large $N$, and the hope is that $\bv^{h,N} \approx \bv$ in such a way that ensures $\widetilde{\mu}_{t+H}^N \approx \mu_{t+H}^N$.
While work on proving precise results of this form is ongoing, we give empirical evidence that the approach works in Sections \ref{sec:chemotaxis}, \ref{sec:burgers}, and \ref{sec:halfmoon}.
At a heuristic level, even though no tangent field describes the evolution of \(\mu_t^N\) on short time scales, there is still a notion of a ``direction'' --- determined by $\mu$ --- in which the \textit{group} of particles is traveling over longer time scales.
The idea is that optimal transport maps allow us to access that direction.

\begin{figure}[tbp]
    \centering{\includegraphics[width=0.8\textwidth]{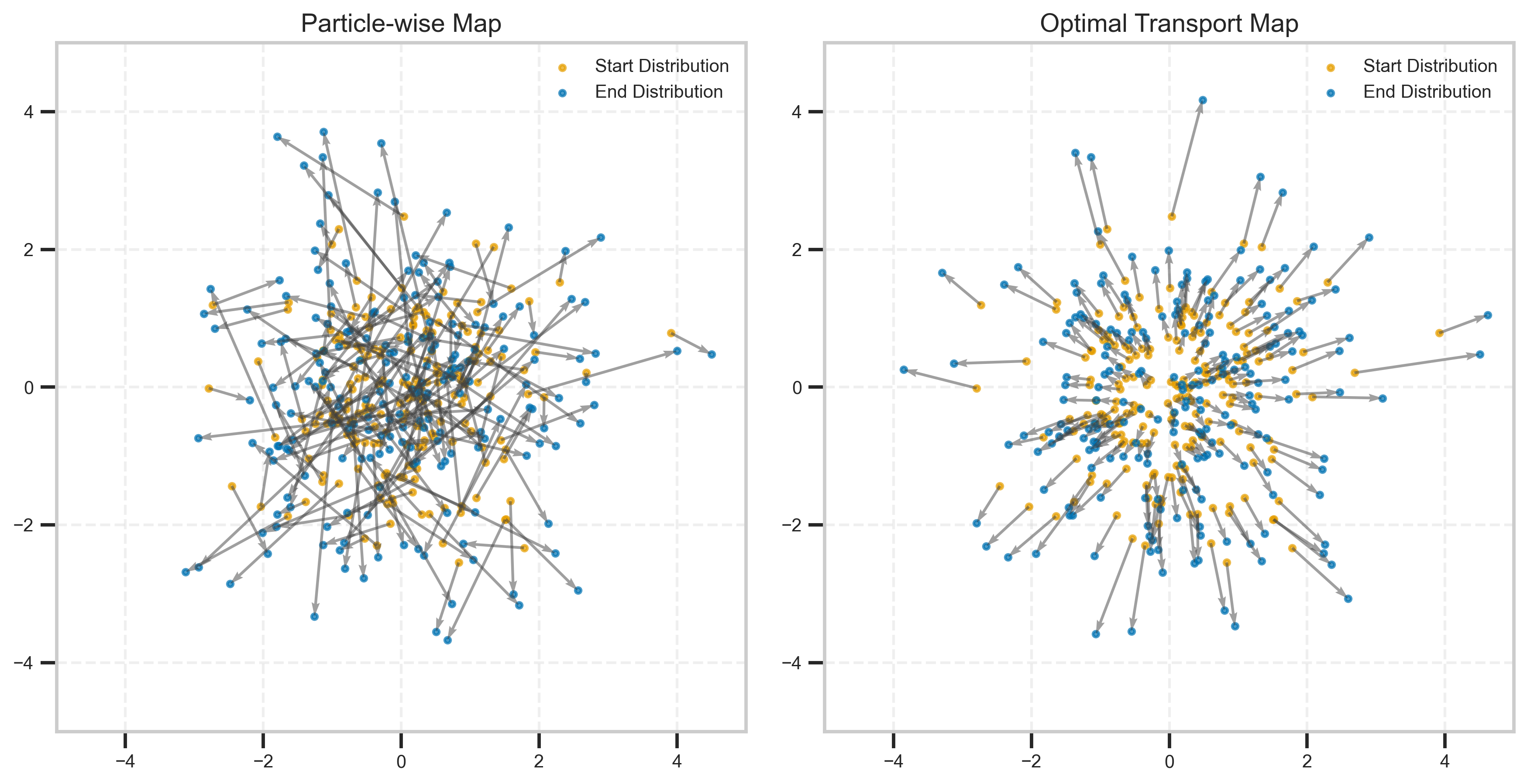}}
    \caption{
    Example of the difference between the optimal transport map and the particle-wise map.
    In the particle-wise image (left), an arrow is drawn to connect each particle's starting position to its ending position.
    In the optimal transport image (right), we compute the OT map between the starting and ending distributions and draw an arrow connecting each particle's starting location to its corresponding output under the OT map.
    }
    \label{fig:miscfigs-OT-vs-Particles}
\end{figure}

As an illustrative example, suppose the $x_i$ are independent, standard $d$-dimensional Brownian motions, the smooth curve ``in the background'' is $t \mapsto \mu_t \coloneqq N(0,tI_d)$, and the oracle is given by an Euler--Maruyama simulation of Gaussian diffusion:
\begin{equation}\label{eqn:discrete-diffusion-simulator-oracle}
    E_h(\nu) = \frac1N\sum_{i=1}^N\delta_{y_i+\sqrt{h} Z_i},
\end{equation}
where $\nu = \frac1N\sum_{i=1}^N \delta_{y_i}$ and $Z_1,\ldots,Z_N \sim \cN(0,I_d)$ are i.i.d.
By Theorem \ref{thm:vectorfields}\ref{item:wp-first-order-Taylor}, the tangent field $\bv$ of $\mu$ is given by
\begin{equation}\label{eqn:gaussian-true-vf-2}
        \bv_t(x) = \lim_{h\to 0}\frac{T_{\mu_t}^{\mu_{t+h}}(x) - x}{h} = \lim_{h\to 0}\frac{\sqrt{\frac{t+h}{t}}-1}{h}x = \frac{x}{2t}.
    \end{equation}
Now, see Figure \ref{fig:miscfigs-OT-vs-Particles}, in which $d=2$.
In the first image, an arrow is drawn from $x_i(1)$ to $x_i(1)+Z_i$, and the arrows are expectedly chaotic.
In contrast, the arrows in the second image are drawn from $x_i(1)$ to $x_i(1) + \bv_1^{1,N}(x_i(1))$ in accordance with our measure-valued Euler-type method, which reveals the obvious structure of the evolution:
The bulk is drifting away from the center in all directions, which agrees with what \eqref{eqn:gaussian-true-vf} suggests.

\subsection{Choosing an Appropriate Step Size}\label{subsec:discrete-choosing-step-size}

In the ``smooth'' case, the error of the approximation $\mu_{t+H} \approx (\id + H\bv_t^h)_{\#}\mu_t$ generically decreases as \(h\) decreases.
In the discrete case, however, the roughness of our particles' evolutions introduces a competing need to keep \(h\) sufficiently large.
To explain why, recall from Proposition \ref{prop:permutation} that optimal transport maps in the discrete case are given by permutations.
Assuming (as we always will) that the particles $\{x_i\}_{i=1}^N$ evolve continuously in time, if $h$ is small enough, then the relevant permutation becomes the identity.
More specifically, if $h$ is very small, then running our micro-scale simulation for time $h$ initialized at $x_i(t)$ will result in a particle $y_i^h \approx x_i(t+h)$ that is very close to $x_i(t)$.
In this case, the map $x_i(t) \mapsto y_i^h$ will optimally transport $\mu_t^N$ to $E_h(\mu_t^N)$.
Therefore,
\begin{align*}
    \bv_t^{h,N}(x_i(t)) & = \frac{T_{\mu_t^N}^{E_h(\mu_t^N)}(x_i(t)) - x_i(t)}{h}  = \frac{y_i^h - x_i(t)}{h} \approx \frac{x_i(t+h)-x_i(t)}{h}, \\
    \widetilde{\mu}_{t+H}^N &  = \frac1N\sum_{i=1}^N \delta_{x_i(t) + H\bv_t^{h,N}(x_i(t))} \approx \frac1N\sum_{i=1}^N \delta_{x_i(t) + H\frac{x_i(t+h)-x_i(t)}{h}}.
\end{align*}
Since $x_i$ is not necessarily differentiable at $t$, $x_i(t) + H\frac{x_i(t+h)-x_i(t)}{h}$ is not necessarily a good approximation of $x_i(t+H)$, so $\widetilde{\mu}_{t+H}^N$ is not necessarily a good approximation of $\mu_{t+H}^N$ if $h$ is too small.

To empirically verify the intuition that small choices of \(h\) result in poor approximations, we return to the case of Brownian motion.
Specifically, suppose again that the $x_i$ are independent Brownian motions in $\R^d$, $\mu_t = \cN(0,tI_d)$, $\bv_t(x) = \frac{x}{2t}$, and the oracle is given in terms of an Euler--Maruyama simulation.
We have observed experimentally that the approximation error \(W_2(\mu_{t+H}^N, \widetilde\mu_{t+H}^N)^2\) is strongly correlated with the error \(\|\bv_t - \bv_t^{h,N}\|_{L^2(\mu_t^N)}^2\) in the approximation of the underlying tangent field (Figure \ref{fig:mixingtime-diffusion-correlation-errors}), so we focus the remaining discussion on the latter quantity for computational convenience.

\begin{figure}[t]
    \center{\includegraphics[width=3in]{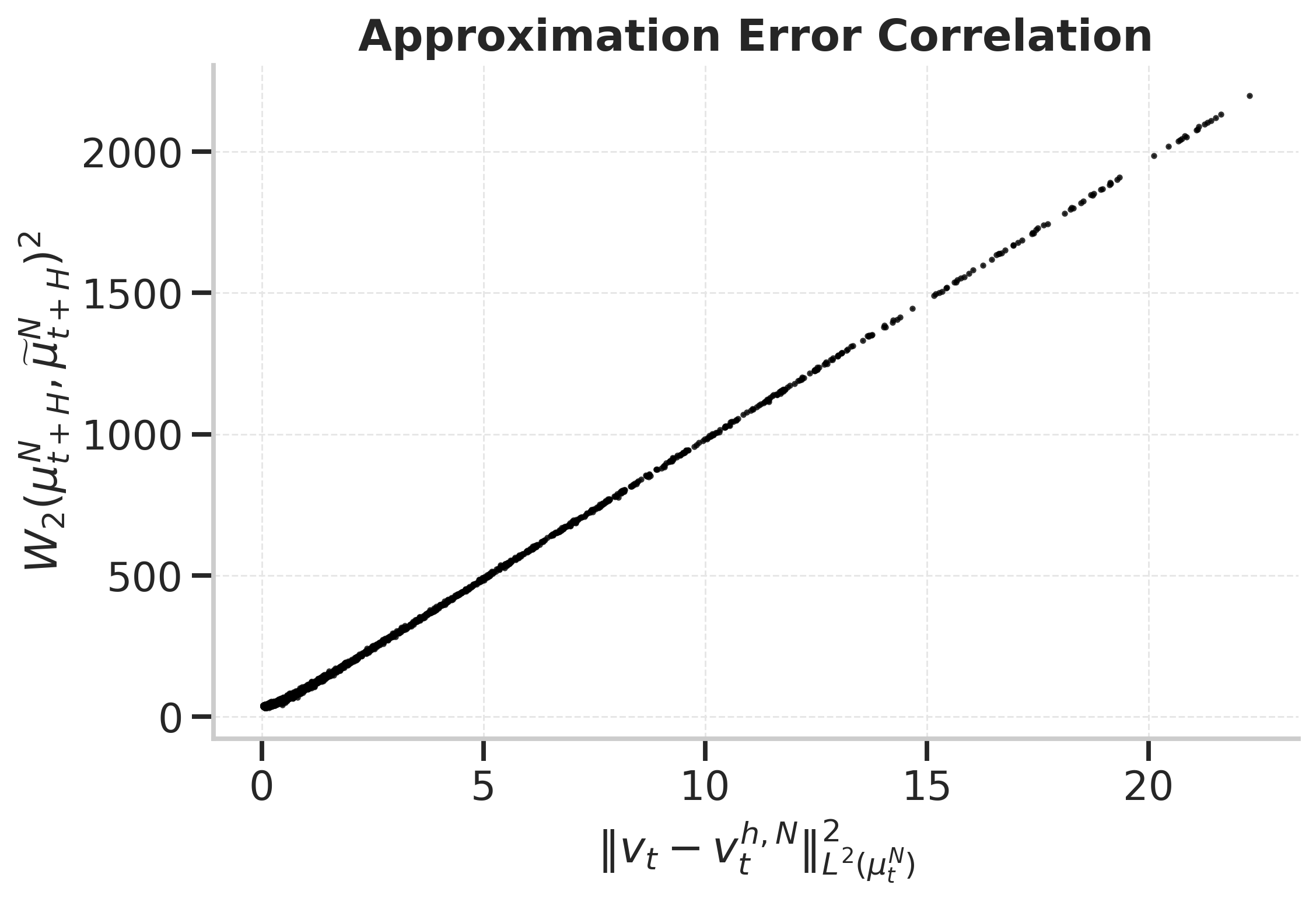}}
    \caption{
        Vector field approximation error \(\|\bv_t - \bv_t^{h,N}\|_{L^2(\mu_t^N)}^2\) versus Euler step approximation error \(W_2(\mu_{t+H}^N, \widetilde\mu_{t+H}^N)^2\) for Brownian motion with \(t = 1\) and \(H = 99\) (\(h\) and \(N\) vary)
    }
    \label{fig:mixingtime-diffusion-correlation-errors}
\end{figure}

\begin{figure}[t]
    \begin{subfigure}[t]{0.49\textwidth}
        \center{\includegraphics[width=\textwidth]{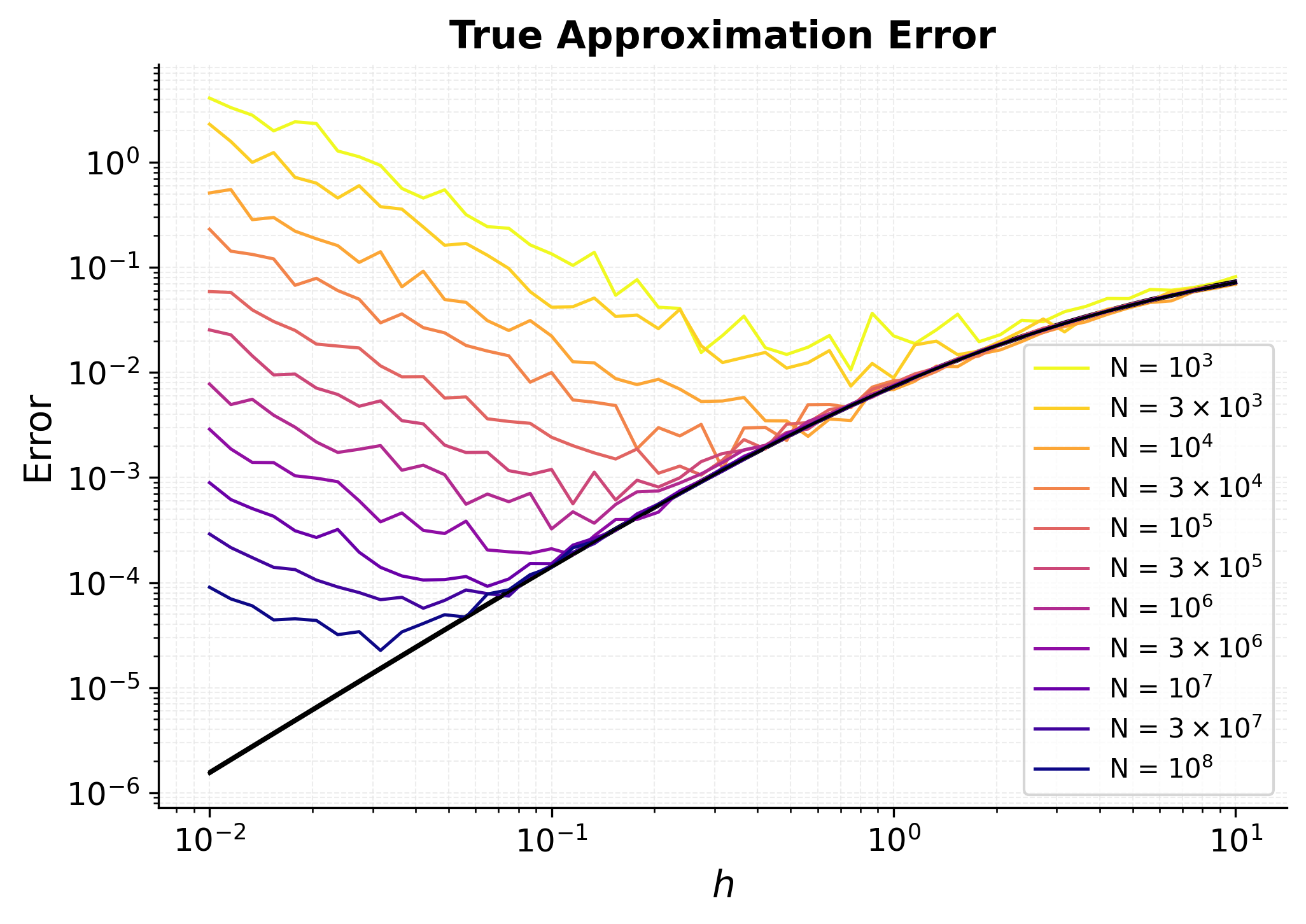}}
        \caption{True approximation error: The black curve is the theoretical finite difference error \(\|\bv_t - \bv_t^h\|^2\) versus the step size \(h\).
        The colored curves are the discrete approximation errors \(\|\bv_t - \bv_t^{h,N}\|^2\) for various values of \(N\).}
        \label{fig:mixingtime-diffusion-vf-error}
    \end{subfigure}
    \hfill
    \begin{subfigure}[t]{0.49\textwidth}
        \center{\includegraphics[width=\textwidth]{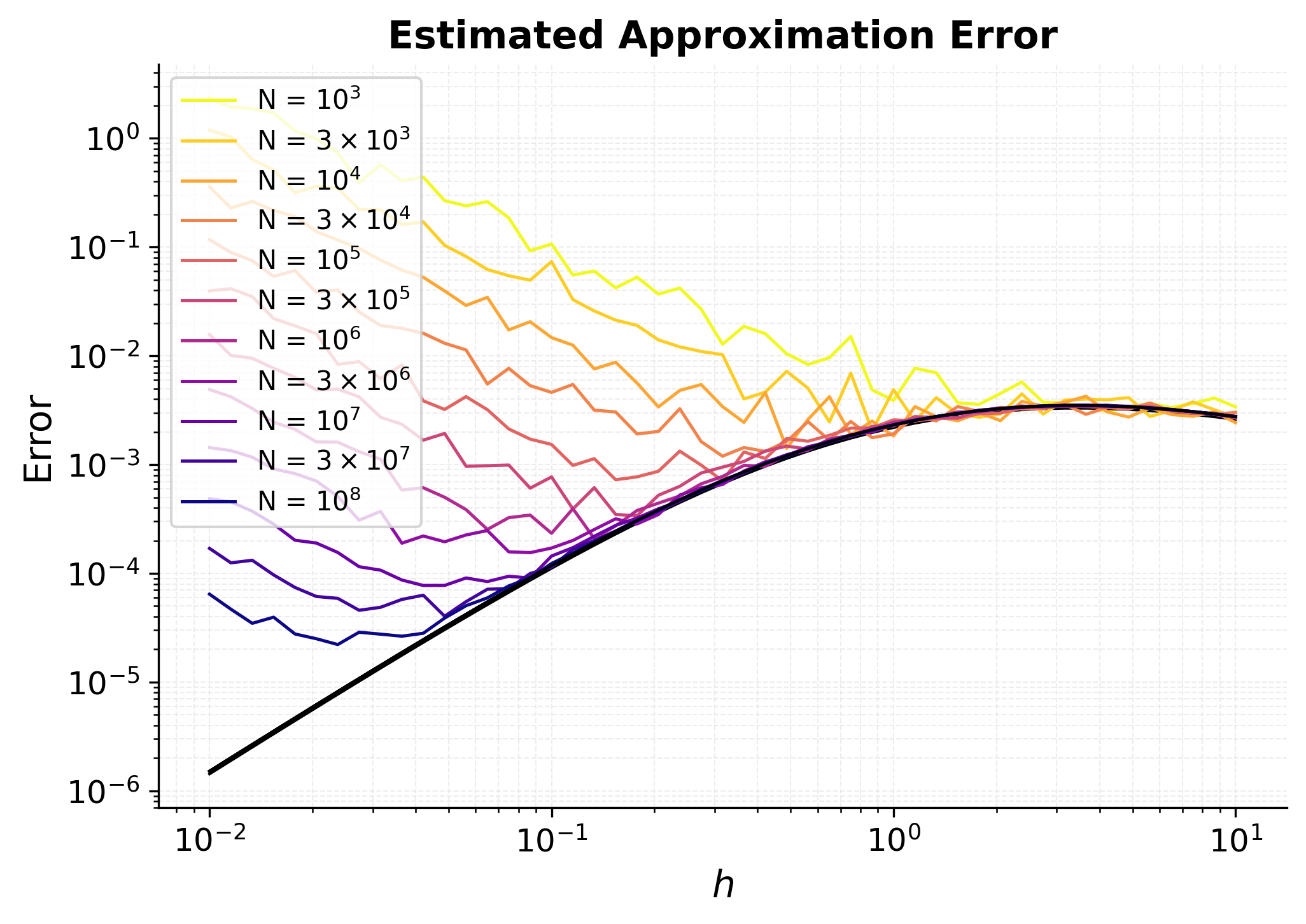}}
        \caption{Estimated approximation error: The black curve is the theoretical finite difference error \(\|\bv_t^h - \bv_t^{2h}\|^2\) versus the step size \(h\).
        The colored curves are the discrete approximation errors \(\|\bv_t^{h,N} - \bv_t^{2h,N}\|^2\) for various values of \(N\).}
        \label{fig:mixingtime-diffusion-vf-difference}
    \end{subfigure}
    \caption{Results of vector field error computations for Brownian motion and comparison with an estimated error using the difference between consecutive approximations.}
\end{figure}

There are two distinct sources of error in the tangent field approximation.
The first is the error in the approximation of a derivative by a finite difference with step \(h>0\); the second is the error in using a discrete optimal transport map to approximate a continuous one.
As suggested above, the first should dominate for large values of \(h\), and the second should dominate for small values of \(h\).
To isolate their effects, we also consider the finite-difference approximation of the tangent field $\bv$,
\[
    \bv_t^h(x) = \frac{T_{\mu_t}^{\mu_{t+h}}(x)-x}{h} = \frac{\sqrt{\frac{t+h}{t}}-1}{h}x,
\]
and compute \(\|\bv_t - \bv_t^h\|_{L^2(\mu_t^N)}\) as a proxy for the error caused by the finite-difference approximation.
In Figure \ref{fig:mixingtime-diffusion-vf-error}, we plot \(\|\bv_1 - \bv_1^h\|_{L^2(\mu_1^N)}\) in black and $\|\bv_1 - \bv_1^{h,N}\|_{L^2(\mu_1^N)}$ for different values of \(h,N\) in various other colors.
We note that for each value of \(N\), there is an approximately optimal value of \(h\) that minimizes the error and that this optimal value of \(h\) decreases as \(N\) increases.
We also see that for very small values of \(h\), the error is very large (as predicted above) and that for larger values of \(h\), the discrete approximation is about as good as the true finite difference approximation.
This indicates that, indeed, as the step size increases, the optimal transport map better captures the distribution's bulk behavior, and the limiting factor on choosing \(h\) too big is the fact that we are approximating a derivative using a finite difference.
While these experiments do not precisely formalize this behavior, they corroborate the assertion that there should be a ``happy-medium'' value of \(h\) that decreases as the number of samples increases.

We expect this phenomenon to appear in most random particle systems one would be interested in studying using optimal transport.
Indeed, there is evidence that this phenomenon is due to a gap in the eigenvalue spectrum of the macroscopic, coarse equation that governs the dynamics ``in the background'' \cite{GK03a,RGK03}.
However, as always, we are considering the setting in which we do not know these ``underlying'' dynamics, and so we cannot \emph{analytically} determine an appropriate value for \(h\).
As such, one might hope to \emph{empirically} determine the optimal choice for \(h\) to use in the prediction algorithm.
Unfortunately, the previous paragraph's analysis was only possible because we had explicit a priori knowledge of $\bv$, which defeats the purpose of approximating it.
A different approach is therefore needed to determine the optimal step size \(h\).
To this end, we perform another experiment to assess the viability of a simple intrinsic criterion: how much the approximation of the vector field changes after another step of size \(h\), i.e., \(\|\bv_1^{h,N} - \bv_1^{2h,N}\|_{L^2(\mu_1^N)}\).
In Figure \ref{fig:mixingtime-diffusion-vf-difference}, we compare \(\|\bv_1^{h,N} - \bv_1^{2h,N}\|_{L^2(\mu_1^N)}\) with \(\|\bv_1^h - \bv_1^{2h}\|_{L^2(\mu_1^N)}\), the ``normal'' difference between two consecutive finite-difference approximations.
While the correspondence is not perfect, the plots are remarkably similar, which suggests the method's potential use.

\begin{rmk}\label{rmk:burst-period-vs-h}
    On a practical note, if the step size of the micro-scale simulator is fixed (e.g., as it is in the model used in Section \ref{sec:chemotaxis}, see Remark \ref{rmk:chemotaxis-step-size-limit}), then one can effectively increase \(h\) by taking \(E_h\) to be the result of applying \emph{multiple} micro-scale steps in succession.
    In that case, we think of \(h\) as a ``burst period'' because we do a short burst of successive micro-scale steps.
    This means we always have the flexibility to increase the effective value of \(h\).
    Thus, it still makes sense to discuss how to tune \(h\) appropriately, even if the micro-scale simulator itself has a fixed time step.
\end{rmk}

In our experiments in Sections \ref{sec:chemotaxis}, \ref{sec:burgers}, and \ref{sec:halfmoon}, we fix this ``burst period'' as a parameter based on data from previous experiments; however, one could easily imagine adaptations to the algorithm in Section \ref{subsec:discrete-algorithm} that dynamically adjust \(h\) using this heuristic.

   %
\subsection{Additional Considerations}\label{subsec:discrete-additional-considerations}

\paragraph{Statistical Averaging}
When working with random particle systems, it is often advantageous to replace the single finite difference $\bv_t^h$ with an average of many finite differences.
There are many ways to do this, but perhaps the simplest is to average finite differences computed with a range of different time steps over some interval.
Also, if we take the middle of this interval to be the time of interest, then taking an average of both forward and backward differences becomes equivalent to taking an average of centered differences, which are typically more accurate numerically.
Explicitly, for a fixed time step \(h\) and discrete times \(t_0 - kh \eqdef t_{-k}, \dots, t_0, \dots, t_k \defeq t_0 + kh\), we set
\begin{align*}
    \bv_{t_0}^{h,N,k}(x_i(t_0)) &\defeq \frac{1}{2k} \sum_{\substack{j=-k \\ j\neq0}}^k \bv_t^{jh}(x_i(t_0)) \\
    &= \frac{1}{2k}\sum_{\substack{j=-k \\ j\neq0}}^k\frac{T_{\mu_{t_0}^N}^{\mu_{t_j}^N}(x_i(t_0)) - x_i(t_0)}{jh} \\
    &= \frac{1}{k}\sum_{j=1}^k\frac{T_{\mu_{t_0}^N}^{\mu_{t_j}^N}(x_i(t_0)) - T_{\mu_{t_0}^N}^{\mu_{t_{-j}}^N}(x_i(t_0))}{2jh}.
\end{align*}
This vector field is the one we use in practice in the experiments in Sections \ref{sec:chemotaxis}, \ref{sec:burgers}, and \ref{sec:halfmoon}.

\paragraph{Burn-in Period}
To approximate the particle distribution $\mu_t^N$ over a long time, we need to perform many steps of our Euler-type method.
Of course, this involves re-initializing the micro-scale simulation many times, once after each step.
One potential obstacle of this approach is that the approximate distribution after an Euler step will inevitably be slightly different from the true distribution, and running the micro-scale simulation on the approximate distribution might result in behavior that is initially dominated by the approximate distribution's return to a relative equilibrium.
To account for this and ensure the approximated velocity fields capture mostly the long-term dynamics, we allow the approximate distribution to evolve according to the micro-scale simulation for a prescribed ``burn-in'' time \(H_R\) before starting to take samples to approximate the velocity field at time \(t+H+H_R\).
This ensures that any fringe effects from the Euler step are washed away before the next Euler step.
For a more theoretical discussion of this intuitive motivation for the burn-in period, we defer to \cite{GK03a, GKKZ05,ZVGKK12}.

We also note that the burn-in period also helps to avoid the pathologies that can arise from using Euler's method on, e.g., Wasserstein gradient flows \cite{XL24}.
For example, there may be locations with a high density of approximated particles (or even particles with exactly the same approximated location), which can result in irregularities in the subsequent vector field approximation.
By allowing the particle system to ``burn in'' after each approximating step, we ensure that the particles' density is, loosely speaking, smoothed out.
For an illustration of how the burn-in time allows the density to smooth, see Figures \ref{fig:burgers-approx-2i} and \ref{fig:burgers-approx-2ii}.
Empirical tests indicate that this burn-in is essential for our method to work well.
Understanding these benefits theoretically is a subject of ongoing work.

    %
\subsection{Detailed Description of Our Algorithm}\label{subsec:discrete-algorithm}
 
\begingroup
\renewcommand{\arraystretch}{1.25}
\begin{table}[htbp]
    \centering
    \begin{tabular}{|l|l|}
        \hline
        Parameter & Description \\ 
        \hline
        \( h \) & the time of one step of the micro-scale simulation \\ 
        \( k \) & the number of centered differences to average \\ 
        \( H_S \coloneqq kh \) & \emph{half} the length of the derivative-sample time window \\ 
        \(N_T\) & the number of Euler-type prediction steps \\
        \( H \) & the length of each Euler step \\ 
        \( H_R \) & the burn-in time after each Euler step \\ 
        \( \Delta \coloneqq H_S + H + H_R \) & the time difference between Euler steps \\
        \(S\) & the length of start-up time \\
        \(R\) & the length of final recovery time after all Euler steps \\
        \hline
    \end{tabular}
    \caption{Parameter definitions.}
    \label{tab:parameters}
\end{table}
\endgroup

We now present our complete algorithm for approximating the evolution of a particle system given by a particular micro-scale simulation \(E_h\) (e.g., the micro-scale simulation could be an Euler-Maruyama numerical simulation of an SDE as in \eqref{eqn:discrete-diffusion-simulator-oracle}, a biological system as in Section \ref{sec:chemotaxis}, or any other method for evolving particles for a short time).
For a particular initial distribution of particles \(\mu_0\), our algorithm approximates the distribution \(\mu_T^N\) for large \(T\) with all the additional considerations discussed in this section.
We fix values for the parameters as defined in Table \ref{tab:parameters} and then:
\begin{enumerate}
    \item Run the micro-scale simulation for the start-up time \(S\) (by taking \(S/h\) micro-scale steps) to allow the initial distribution \(\mu_0\) to reach a relative equilibrium (if necessary).
    Call the resulting distribution \(\mu_{(-1)+H_R}\).
    \item For each approximating step \(n = 0, \dots, N_T\):
    \begin{enumerate}
        \item Set \(\mu_{(n)-kh}\defeq \mu_{(n-1)+H_R}\), and take \(2k\) micro-scale steps to get distributions
        \[
            \mu_{(n)-kh}^N, \dots, \mu_{(n)}^N, \dots, \mu_{(n)+kh}^N.
        \]
        \item For each \(x_i\) in the support of \(\mu_{(n)}^N\), compute
        \[
            \bv_{(n)}^{h,N,k}(x_i) \defeq \frac{1}{k}\sum_{j=1}^k\frac{T_{\mu_{(n)}^N}^{\mu_{(n)+jh}^N}(x_i) - T_{\mu_{(n)}^N}^{\mu_{(n)-jh}^N}(x_i)}{2jh}.
        \]
        \item Set
        \[
            \mu_{(n)+H}^N \defeq \left(\id + \bv_{(n)}^{h,N,k}\right)_\# \mu_{(n)}^N.
        \]
        \item Take \(H_R/h\) micro-scale steps to allow the distribution to burn in.
        Call the resulting distribution \(\mu_{(n)+H_R}^N\).
    \end{enumerate}
    \item Run the micro-scale simulation for some final recover time \(R\) (by taking \(R/h\) micro-scale steps) to ensure any fringe errors in the approximations are minimized (e.g., that the distribution has indeed reached a true steady state).
\end{enumerate}

While the notation used above is convenient for highlighting the recursive nature of the algorithm, it becomes unwieldy when discussing how to interpret the approximations.
As such, we set the following notation convention:

\begin{nota}
    Let \(\widetilde\mu_t^N\) denote the approximation produced by the algorithm that corresponds to time \(t\).
\end{nota}

More explicitly, if \(t_n \defeq S+n(H_S + H + H_R)+H_S\), then \(\widetilde\mu_{t_n}^N = \mu_{(n)}^N\) is the approximation we use to compute the Euler step, \(\widetilde\mu_{t_n+H}^N = \mu_{(n)+H}^N\) is the approximation immediately after the Euler step, and \(\widetilde\mu_{t_n+H+H_R}^N = \mu_{(n)+H_R}^N\) is the approximation after the burn-in period.
This notation is convenient because \(\widetilde\mu_t^N\) is an approximation of \(\mu_t^N\) for all \(t\) for which the former is defined.


\begin{rmk}\label{rmk:hidden-variables}
    One limitation of this scheme that will become relevant in Section \ref{sec:chemotaxis} is that it only gives a way to approximate future \emph{locations} of particles and not any underlying data.
    If the micro-scale model has any parameters besides the position of the particles (e.g., the current ``underlying state'' of a particle), then these data must also be predicted across the Euler-type predicting step taken in step 2(c) above.
    In general, there is no single method for doing this, but we will discuss ad hoc approaches as they come up in Section \ref{sec:chemotaxis} where we must also predict the state of the bacteria's flagella and other underlying data (see Sections \ref{subsec:chemotaxis-experimental-results} and \ref{subsec:chemotaxis-compare-reinits}).
\end{rmk}

\section{Simulation Study I: Biological Agent-Based Model}\label{sec:chemotaxis}

We now demonstrate the efficacy of our algorithm on a few illustrative example micro-scale models.
The first such model we consider is one of bacterial chemotaxis.


\subsection{Micro-scale Model Description}

The specific model we use is described in detail in \cite{SGOK05,LPSK23}.
For brevity, we omit the precise details of the (see \cite{SGOK05,LPSK23} for more details), but we provide a high-level overview here to illustrate its computational complexity (and hence the desire for an algorithm to predict its evolution).
Each cell's mobility is affected by a chemoattractant profile (specifically, its spatial derivative) which influences the ``rotations'' of the cell's six ``flagella.''
The model computes the probability that each flagellum changes its rotational direction based on changes of the concentration of the protein CheY-P.
The concentration of CheY-P has a time-derivative which is a function variables which evolve according to a joint ODE that itself is a function of the chemoattractant profile.
The rotational directions of the flagella determine the motion of the cell at each step --- if fewer than three are rotating counter-clockwise, the cell ``tumbles,'' meaning it does not move at that time step.
Otherwise, the cell ``swims,'' meaning it moves with constant speed in the direction that it moved in the last time step (or chooses a new random direction if it was tumbling in the last time step).
Then to actually execute one time step, the model updates the locations of the cells that swim for that step, and then it updates all underlying variables according to the prescribed derivatives and the value of the chemoattractant profile at the new location. 

Due to its intricate nature, many more operations are needed to execute one step of this model than, e.g., a standard SDE Euler-Maruyama model.
Thus, our approximation algorithm provides an appreciable computational improvement by allowing us to skip many of these expensive micro-scale steps.

\begingroup
\renewcommand{\arraystretch}{1.25}
\begin{table}[tbp]
\begin{center}
\begin{tabular}{|l|l|}
    \hline
    Parameter & Description \\
    \hline
    $N = 10{,}000$ & the number of bacteria (the number of simulations) \\
    $h = 1/16$ & micro-scale time step \\
    $v_{\text{cell}} = 0.003$ & swimming speed of each bacteria \\
    $S(x) = \cN(\mu=6.5, \sigma=1.35)$ & chemoattractant profile \\
    \hspace{1em} $=(1.35\sqrt{2\pi})^{-1}\exp\left(-\frac{(x-6.5)^2}{2(1.35)^2}\right)$ & \\
    \hline
\end{tabular}
\caption{Parameters for chemotaxis micro-scale simulation.}
\label{tab:chemotaxis-paramaters}
\end{center}
\end{table}
\endgroup

\subsection{Experimental Results}\label{subsec:chemotaxis-experimental-results}

\begin{figure}[tbp]
    \centering
    \includegraphics[width=0.9\textwidth]{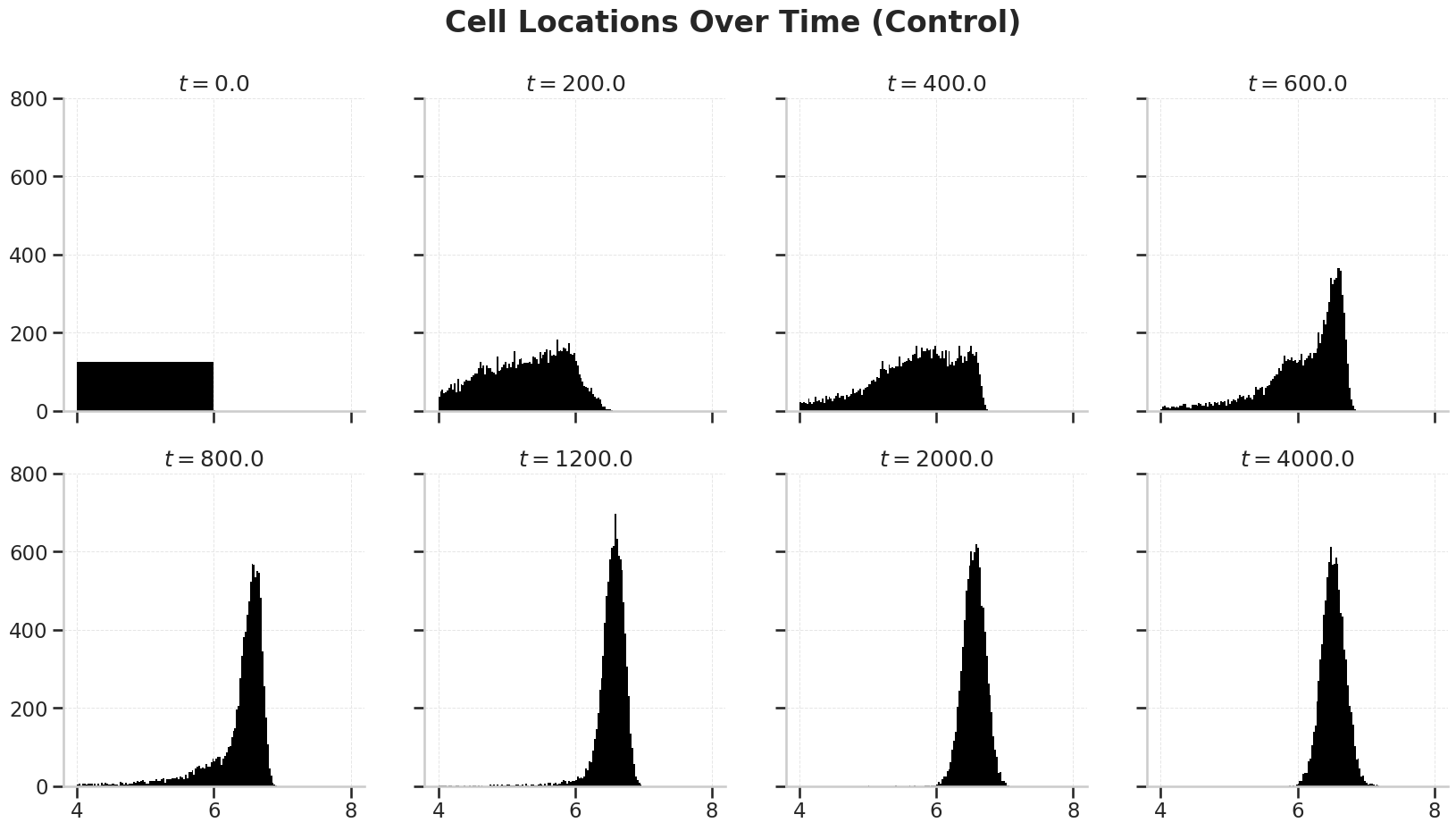}
    \caption{Histograms of bacteria locations at various times \(t\) from \(t=0\) to \(t=4000\) for the control simulation of chemotaxis.}
    \label{fig:chemotaxis-0mastercontrol-1d-1}
\end{figure}

For reference, we run the simulation without prediction to establish a ``control" as the ground truth behavior of the system.
We initialize the cell locations \(x_i(0)\) to be evenly spaced in the interval \([4,6]\), and we use the same parameters as in \cite{SGOK05, LPSK23} with only the changes specified in Table \ref{tab:chemotaxis-paramaters}.
Histograms of bacteria positions at different times in the control simulation are shown in Figure \ref{fig:chemotaxis-0mastercontrol-1d-1}.
Then, to compare with the control data, we perform a simulation using the approximation algorithm with the additional parameters in Table \ref{tab:chemotaxis-addparams}.
Note that we choose these parameters (in particular, the number of centered differences \(k\) --- which effectively gives the length of the ``burst'' period discussed in Remark \ref{rmk:burst-period-vs-h} --- and the length of the Euler-type prediction step \(H\)) according to the heuristics discussed in Section \ref{subsec:discrete-choosing-step-size}.
For simplicity, we choose to fix these parameters at the beginning based on previously-observed data about the micro-scale model, but one could design algorithms that make these choices adaptively based on, e.g., how quickly the computed velocity field is changing --- the criterion discussed at the end of Section \ref{subsec:discrete-choosing-step-size}.
\begingroup
\renewcommand{\arraystretch}{1.25}
\begin{table}[tp]
\begin{center}
\begin{tabular}{|l|l|}
    \hline
    Parameter & Description \\
    \hline
    \(S = 200\) & length of start-up time \\
    \(k = 80\) & number of centered differences to average \\
    \(H = 95\) & length of each Euler-type prediction step \\
    \(H_R = 100\) & burn-in time after each Euler step \\
    \(N_T = 19\) & number of Euler steps \\
    \(R = 0\) & length of final recovery time after last Euler step \\
    \hline
\end{tabular}
\caption{Additional parameters for approximation algorithm with chemotaxis micro-scale simulation.}
\label{tab:chemotaxis-addparams}
\end{center}
\end{table}
\endgroup

To continue simulating after each Euler step, the micro-scale simulator requires not only current bacteria locations but also all microscopic quantities. However, as discussed in Remark \ref{rmk:hidden-variables}, the algorithm described in Section \ref{subsec:discrete-algorithm} predicts only the \emph{locations} of bacteria after each Euler step, resulting in a lack (or missing) of microscopic quantities. Hence, after the predicting time step, the suggested algorithm requires the re-initialization of microscopic quantities, such as $u_1, u_2$, and $s^{(i)}$. Even though some (quantitatively bad but) qualitatively good re-initializations quickly reach slow dynamics, we suggest a better re-initialization scheme based on the history of microscopic quantities and demonstrate a more accurate prediction of bacteria density distribution over multiple Euler~steps.

After each Euler step, we re-initialize the underlying data according to the history-dependent scheme in Table \ref{tab:scheme-comparison}. For example, given data at time \(t_0\), we use the optimal transport approximation scheme to predict the bacteria locations at \(t_0+H\) and re-initialize the microscopic properties at \(t_0+H\) as follows:
\begin{itemize}
    \item Set \(u_1=0\) for all bacteria,
    \item compute \(u_2'(t_0) = \frac{du_2}{dt}\big|_{t=t_0}\) from the prescribed ODE in \cite{SGOK05} to take a forward Euler step
    \(\widetilde u_2(t_0+H) \defeq u_2(t_0) + H\cdot u_2'(t_0)\), and
    \item keep the flagella rotating direction and bacteria moving direction exactly the same as they were at time \(t_0\).
\end{itemize}

\begin{figure}[tbp]
    \centering
    \includegraphics[width=0.9\textwidth]{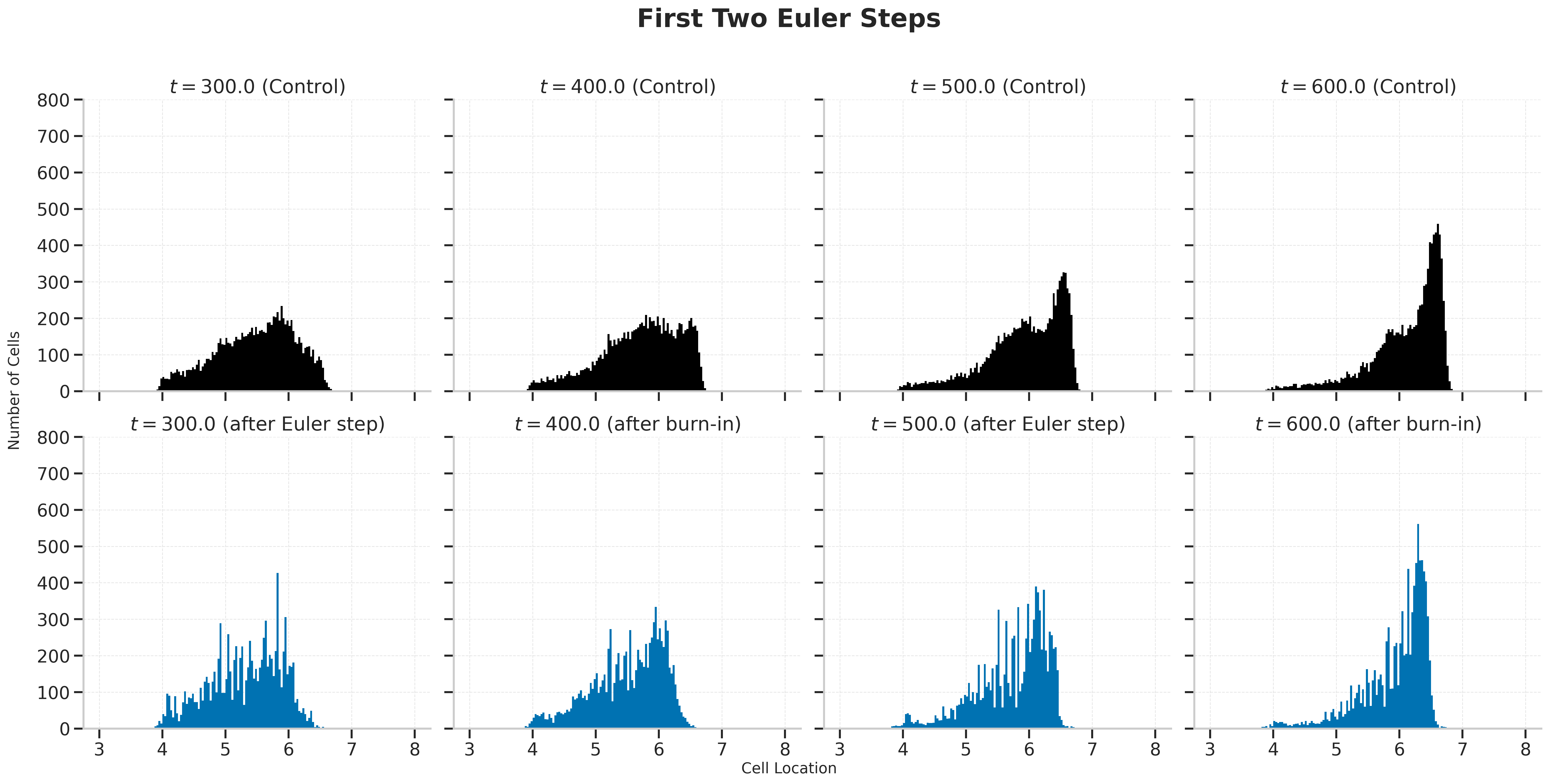}
    \caption{
    Histograms of bacteria locations after first two Euler steps (bottom) compared with corresponding times from control simulation (top).
    We run the micro-scale simulation until \(t=200\) (not pictured), take an Euler step to \(t=300\) (leftmost), burn in until \(t=400\) (second from left), take an Euler step to \(t=500\) (second from right), and burn in until \(t=600\) (rightmost).
    }
    \label{fig:chemotaxis-1approx1-1d-2i}
\end{figure}
\begin{figure}[tbp]
    \centering
    \includegraphics[width=0.9\textwidth]{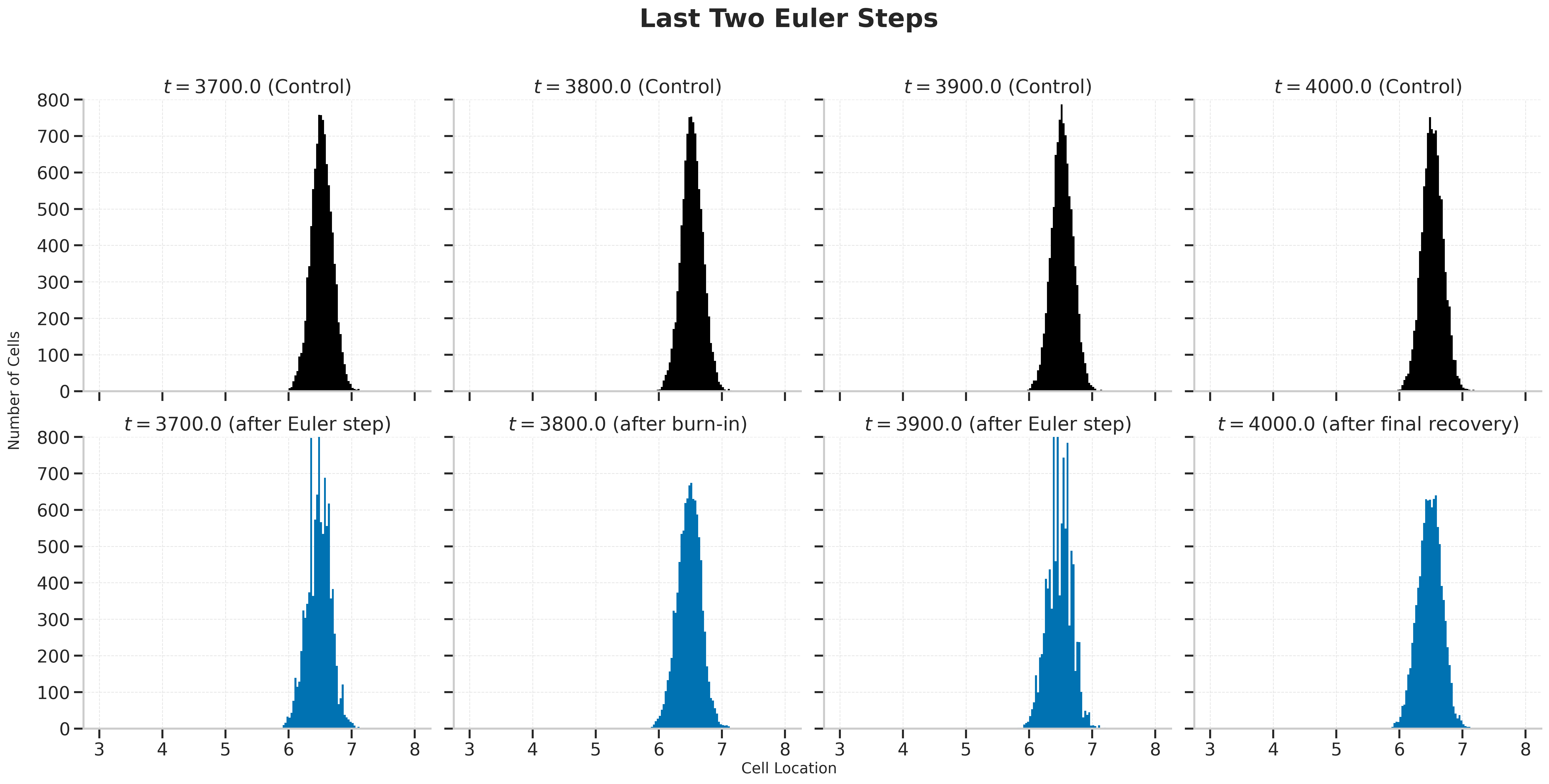}
    \caption{Histograms of bacteria locations after final two Euler steps (bottom) compared with corresponding times from control simulation (top). We take an Euler step from \(t=3600\) (not pictured) to \(t=3700\) (leftmost), burn in until \(t=3800\) (second from left), take an Euler step to \(t=3900\) (second from right), and finally burn in until \(t=4000\) (right-most), the end time.}
    \label{fig:chemotaxis-1approx1-1d-2ii}
\end{figure}

To qualitatively inspect the accuracy of the approximations, we show histograms of the control and approximated distributions of the bacteria locations after the first two and last two Euler steps in Figures \ref{fig:chemotaxis-1approx1-1d-2i} and \ref{fig:chemotaxis-1approx1-1d-2ii}, respectively.
We see that the approximations match reasonably closely with the control distributions at the corresponding times.
In particular, we see that our algorithm approximates the correct steady-state distribution, as evidenced by Figure \ref{fig:chemotaxis-1approx1-1d-2ii} --- the last two Euler steps do not change the approximated distribution, meaning the velocity field is approximately zero, which is what it should be for a distribution in a steady state.

To quantitatively assess the accuracy of our approximations, we compute the Wasserstein distance between the control distribution \(\mu_t^N\) and the approximated distribution \(\widetilde\mu_t^N\).
A plot of the distance \(W_2(\mu_t^N,\widetilde\mu_t^N)\) versus time $t$ (for the discrete set of times from our experiments) is shown in Figure \ref{fig:chemotaxis-1compare-schemes-1}.
We can also visualize how well the approximations track the true curve \(t\mapsto \mu_t^N\) by embedding into a common linear space and performing a dimension reduction.
Specifically, we (1) fix a common (discrete) reference measure \(\sigma\) (in our case, by choosing \(N = 10{,}000\) independent samples of a standard Gaussian); (2) compute optimal transport maps \(T_\sigma^{\mu_t^N}\) and \(T_\sigma^{\widetilde\mu_t^N}\); (3) vectorize \(T_\sigma^{\mu_t^N}\) and \(T_\sigma^{\widetilde\mu_t^N}\) by interpreting them as \(N\times d\) matrices, where \(N\) is the number of particles in \(\sigma\) and \(d\) is the ambient dimension (two in our case); and (4) perform PCA to project onto a two-dimensional subspace.
The results are shown in Figure \ref{fig:chemotaxis-1compare-schemes-0}.

\begin{rmk}\label{rmk:chemotaxis-step-size-limit}
    We note that the micro-scale model used in this example actually breaks down for \(h\) much larger than our choice of \(1/16\).
    For example, with the rest of the model parameters fixed, the computation of the underlying variables as described in \cite{SGOK05} fails badly for \(h = 1/4\).
    In particular, there is an assumption in \cite{SGOK05} that \(h\cdot k_{\pm}\ll 1\), and by choosing \(h = 1/4\), the values of \(k_{\pm}\) over time become many orders of magnitude larger than 1, which means this assumption is grossly false.
    Through trial and error, we determine that \(h = 1/16\) is about the largest time step that preserves the above inequality for the entire simulation.
    This means that we really cannot improve the efficiency of the micro-scale model by making the time step larger --- the time step must be kept sufficiently small to preserve the model of the micro-scale dynamics, but we still need to simulate the model for a long time to observe the steady-state behavior, which is exactly the motivating archetype of the paper.
    For more examples of models with a similar property, and a more detailed discussion of the resulting tradeoff between accuracy, stability, and computational cost, we defer the reader to \cite{GK03a,RGK03}.
\end{rmk}

\subsection{Computational Improvement}\label{subsec:chemotaxis-computational-improvement}

As we briefly discuss in the introduction, we are not interested in the precise implementation of the micro-scale model, and instead we are interested in demonstrating that our approach does indeed reduce the number of micro-scale steps needed to accurately take a ``macro'' time step.
Because of this, we illustrate the computational improvement of our method by computing its reduction in required micro-scale steps.
With the parameters defined above, our algorithm requires 11/20 as many micro-scale steps as the equivalent control simulation (10 seconds of velocity-field-approximation and 100 seconds of burn-in time for every 200 seconds of simulated time).
Importantly, the length of each micro-scale burst to approximate the velocity field is 1/10 as long as the corresponding Euler-step, and so ignoring burn-in, our ``macro-scale timestepper'' only requires 1/10 as many micro-scale steps as running the simulation to the equivalent time.
Thus, the primary way the overall improvement could be increased is by decreasing the burn-in time.
We have experimentally observed that with our choice of re-initialization, the burn-in time required for the underlying variables to return to relative equilibrium scales roughly with the size of the Euler step.
This means that achieving a better ratio of required micro-scale steps likely requires a better approach to re-initializing the underlying variables.
We have seen in further experiments (omitted for brevity) that the burn-in time can be reduced considerably if the underlying variables are re-initialized perfectly, e.g., by using data from the control simulation.
This suggests that the current barrier to more significant improvements is the quality of re-initializations, not the quality of the distribution approximations.
We demonstrate the importance of the re-initialization scheme in Section \ref{subsec:chemotaxis-compare-reinits}.

\subsection{Comparison with Alternative Methods} \label{subsec:chemotaxis-alternative-methods}

\paragraph {Particle-wise approximation}
To demonstrate the value of using optimal transport in our algorithm, we compare our method with a similar method that uses a different approach to learning the underlying dynamics.
Our method uses a short burst of micro-scale steps to gather data that is used to approximate a velocity field, so a natural question is whether the same data could be used more effectively; i.e., is the apparent success of our method due to optimal transport, or is it simply the case that there is enough data in the short burst that any reasonable approach would be just as successful?

One extremely naive approach is to attempt to approximate the velocity field without using optimal transport.
Specifically, we can perform the finite-difference calculations on the individual cells:
\[
    x_i^{[1],h,k}(t) \coloneqq \frac{1}{k}\sum_{j=1}^k\frac{x_i(t+jh) - x_i(t-jh)}{2jh}, \qquad i=1,\ldots,N.
\]
Effectively, this attempts to approximate \(x_i'(t)\) for each individual cell, which we could then use in an Euler-type approximation
\[
    \widetilde{x}_i(t+H) \coloneqq x_i(t) +  Hx_i^{[1],h,k}(t), \qquad i=1,\ldots,N.
\]
As explained in Section \ref{subsec:discrete-choosing-step-size}, this approach should result in a much worse approximation of the distribution's overall evolution.
Our experiments agree with this prediction.
We run exactly the same approximation algorithm with the same parameters as above, with the only change being that we use the particle-wise finite differences as the approximated velocity field rather than the OT finite differences. 
In Figure \ref{fig:chemotaxis-1compare-schemes-1}, we see that the \(W_2\) distance between the control and approximated distributions is substantially higher, and in Figure \ref{fig:chemotaxis-1compare-schemes-0}, we see that the approximations fail to closely track the control distribution's trajectory.
This is a naive approach, but it validates the intuition that optimal transport really is doing something nontrivial in the algorithm.

\paragraph{JKOnet* approximation}
A much more reasonable alternative to our method is to attempt to learn an SDE that approximates the observed behavior.
That is, instead of using the short burst of micro-scale steps to approximate a velocity field, we could use that data to learn an SDE whose behavior closely mimics the observed cell behavior from the micro-scale model.
With this SDE in hand, we could hope to approximate the cells' distribution a long time in the future by running the learned SDE for the equivalent time.
Since simulating SDEs is computationally cheap, this would also give significant computational improvements by replacing many expensive micro-scale steps with cheap SDE steps.

To evaluate whether this approach is viable, we perform a similar experiment (again with the same parameters as above, including start-up, burn-in, and final recovery), except we use the data from the \(2k+1\) micro-scale steps in the micro-scale burst to learn an SDE rather than approximate the velocity field.
To learn the SDE, we use JKOnet* \cite{TLGD24}, which is a particularly successful method for learning SDEs given particle data at successive times.
We use this data to learn the potential, interaction, and internal energies of an approximating underlying SDE, and use those learned parameters to simulate an SDE with initial particle positions equal to the locations of the cells in the last cell distribution from the micro-scale burst.
We simulate this SDE for the same length of time as our ``Euler-type prediction step'' (\(H = 95\) in the case of the experiment above), and reinitialize our micro-scale chemotaxis model with cells at the locations of the particles after the SDE simulation.
We use the same re-initialization scheme for all the underlying data, meaning the only difference between our method and the new method using JKOnet* is how we predict the cell locations at time \(t+H\) in each Euler-type step.

\begin{figure}[tp]
    \begin{subfigure}[t]{0.49\textwidth}
        \center{\includegraphics[width=\textwidth,height=2in]{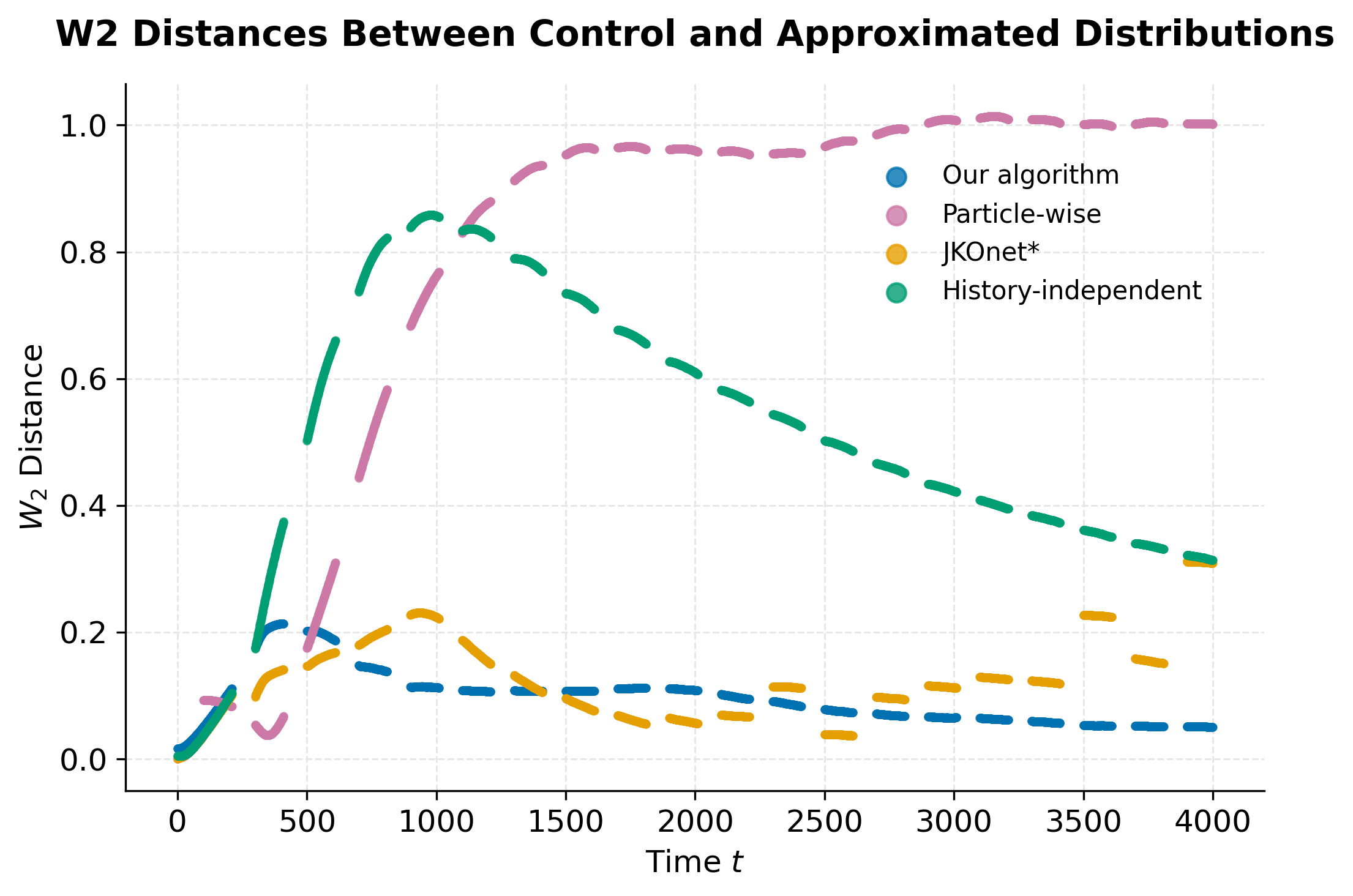}}
        \caption{2-Wasserstein distance between control distribution and approximated distributions versus time.}
        \label{fig:chemotaxis-1compare-schemes-1}
    \end{subfigure}
    \hfill
    \begin{subfigure}[t]{0.49\textwidth}
        \center{\includegraphics[width=\textwidth,height=2in]{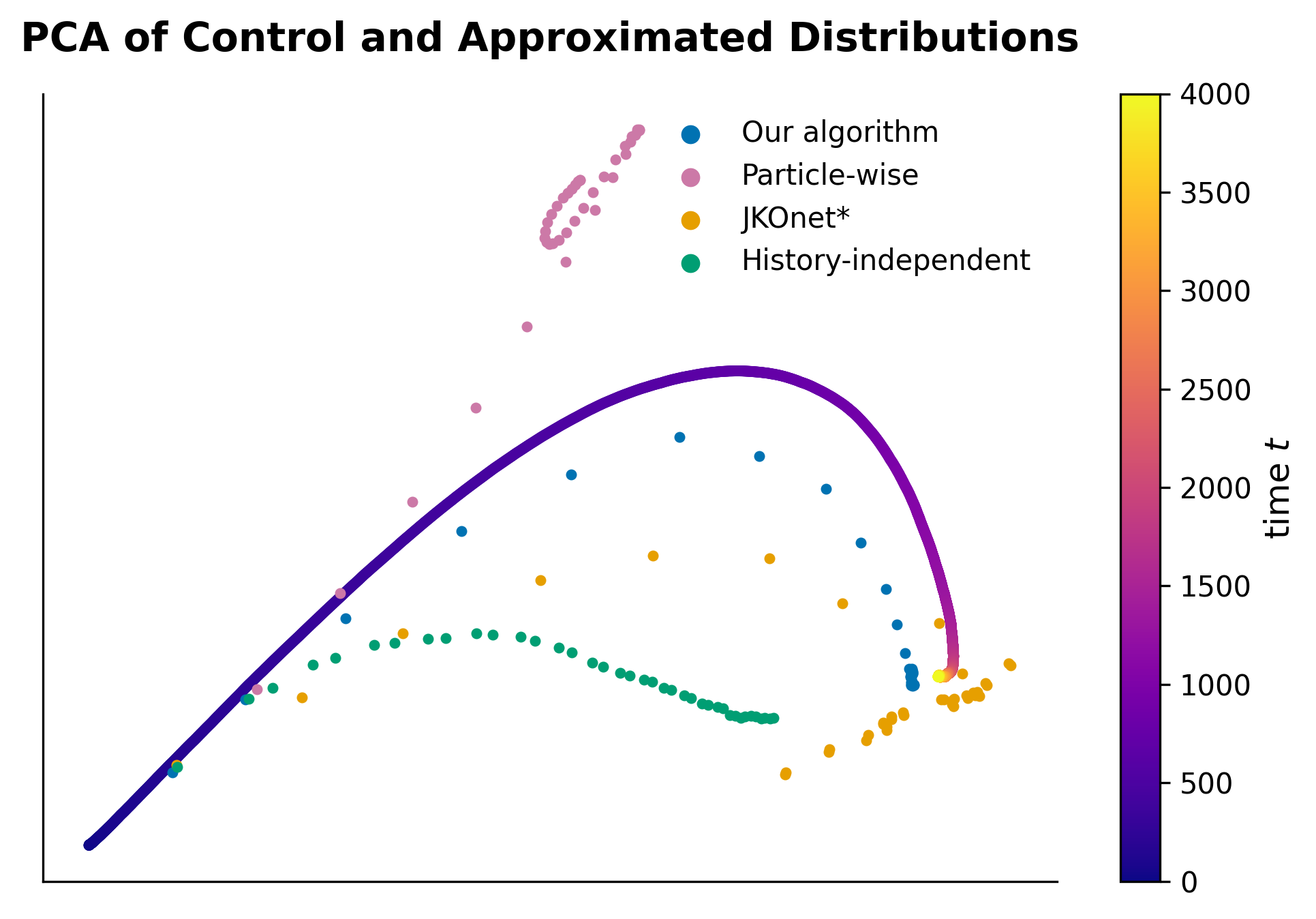}}
        \caption{Projections of distributions onto the two principal components of the control distribution's evolution over time.
        The solid rainbow curve shows the control distribution, where the color represents time according to the color bar at the right.
        }
        \label{fig:chemotaxis-1compare-schemes-0}
    \end{subfigure}
    \caption{Comparisons between control distributions and approximations.
    The four approximation schemes are as follows:
    the blue is our algorithm, as described in Section \ref{subsec:chemotaxis-experimental-results};
    the purple is using the particle-wise finite differences, as described in the first half of Section \ref{subsec:chemotaxis-alternative-methods};
    the green is using the JKOnet* method for learning and simulating an SDE, as described in the second half of Section \ref{subsec:chemotaxis-alternative-methods};
    the orange is using our algorithm with the history-independent re-initialization scheme, as described in Section \ref{subsec:chemotaxis-compare-reinits}.
    }
    \label{fig:chemotaxis-compare-approximations}
\end{figure}

In Figure \ref{fig:chemotaxis-1compare-schemes-1}, we produce a comparison of the \(W_2\) distance between the control and approximated distributions over time for the two different methods, and in Figure \ref{fig:chemotaxis-1compare-schemes-0}, we perform PCA to visualize how well the approximations track the trajectory of the control evolution.
We see that the JKOnet* approach performs substantially worse than our method.
This is because learning an SDE typically requires simulation data from many more times (and in particular, from a wide range of times during the evolution).
In our case, we only have access to simulation data from the short micro-scale burst of \(2k+1\) steps, which as a time interval is a small fraction of the total length of time we wish to simulate.
It is this restriction to only \emph{local} data that makes the SDE-learning approach so prone to errors in this context.
This experiment shows how our method provides a tradeoff between the need for data and the generality of the object being learned --- less data is required if one is only interested in understanding how a particular distribution evolves according to the micro-scale model rather than if one hopes to learn a model that captures the global dynamics of the model.
If we wanted an SDE that could quickly simulate the evolution of \emph{any} initial distribution, then we could run many more micro-scale steps and then use JKOnet* to learn an appropriate SDE.
However, our goal is not to learn a model for how the model \emph{would} act on \emph{any} distribution --- we are only interested in developing a method for taking a \emph{particular} initial distribution and predicting its evolution using as few micro-scale steps as possible.

\subsection{Comparison with Different Re-initializations}\label{subsec:chemotaxis-compare-reinits}

In section \ref{subsec:chemotaxis-computational-improvement}, we briefly allude to the importance of the method by which we re-initialize the underlying variables across each Euler-type step.
To demonstrate the efficacy of our suggested re-initialization described in Section \ref{subsec:chemotaxis-experimental-results}, we also perform a history-independent re-initialization as described in \cite{SGOK05}.
Table \ref{tab:scheme-comparison} shows the different parameters of re-initialization between the two schemes.
In Figure \ref{fig:chemotaxis-1compare-schemes-1}, we provide a comparison of the \(W_2\) distance between the control and approximated distributions over time for the two different approximation schemes, and in Figure \ref{fig:chemotaxis-1compare-schemes-0}, we perform PCA to visualize how well the approximations track the trajectory of the control evolution.

\begingroup
\renewcommand{\arraystretch}{1.5}
\begin{table}[ht]
    \centering
    \begin{tabular}{|c|l|l|}
        \hline
        \multicolumn{1}{|c|}{Parameters} & \multicolumn{1}{|c|}{History-Dependent Scheme} & \multicolumn{1}{|c|}{History-Independent Scheme} \\
        \hline 
        \(u_1\) & \(\widetilde u_1(t_0+H) \defeq 0\) & \(\widetilde u_1(t_0+H) \defeq 0\) \\
        \(u_2\) & \(\widetilde u_2(t_0+H) \defeq u_2(t_0) + H\cdot u_2'(t_0)\) &
        \(\widetilde u_2(t_0+H) \defeq f(S(\tilde x(t_0+H)))\) \\
        \(s^{(i)}\) & \(\tilde s^{(i)}(t_0+H) \defeq s^{(i)}(t_0)\) & \(\tilde s^{(i)}(t_0+H) \defeq 0\) (fixed) \\
        \hline
    \end{tabular}
    \caption{Parameters of different re-initialization schemes. The ``history-dependent'' scheme (as suggested in this paper) retains the flagella states from the micro-scale simulation immediately before the Euler step, and it attempts to predict \(u_2\) via an Euler forward step. The ``history-independent'' scheme (described in \cite{SGOK05}) sets a fixed number of flagella to rotate counterclockwise (here, 0), and \(u_2\) is set solely as a function of the predicted location of bacteria.}
    \label{tab:scheme-comparison}
\end{table}
\endgroup
    
The first Euler step shows similar results between the two experiments, but the history-independent scheme becomes (qualitatively good but) ``relatively" inaccurate by the end of the first burn-in period. This is because the history-independent method re-initializes the microscopic properties with a ``fixed'' value regardless of current conditions, resulting in a large burn-in time to reach slow dynamics, while the history-dependent scheme uses the current microscopic properties to re-initialize after the next Euler step. In other words, the underlying variables (particularly \(u_2\) and the flagella rotating direction) could not respond quickly enough to the fixed re-initialization. Furthermore, even though the distributions look the same after the first Euler step, the two methods still yield different subsequent microscopic behavior, leading to slightly different slow dynamics.

The difference in the two methods' approximation accuracy suggests that the underlying microscopic properties are important components in the micro-scale simulation. Hence, different settings of the properties result in different overall distribution behavior. In the language of Section \ref{sec:otvfs}, the tangent field that describes the distribution's evolution is somehow a function of these variables, not merely a function of position, so the method we choose to predict them across Euler steps is critical to our capability to approximate the evolution using optimal transport maps.
Like all mathematical models, choices about which prediction methods should be used depend entirely on the needs of the particular application. Unfortunately, these choices are mostly ad hoc and context-specific, so there is little to say in general about how they impact the accuracy of the overall approximation algorithm.

\section{Simulation Study II: Partial Differential Equation} \label{sec:burgers}

To further demonstrate the ubiquity of our approach, we now perform our algorithm on a particle model which simulates the well known Burgers' equation
\[
    \del_t\mu = \mu\del_x\mu+\nu\del_{xx}\mu.
\]
The micro-scale model we use is described in detail in \cite{AK21}, but we provide a brief description here.
At each discrete time step of size \(h\) (i.e., \(t_{j+1} = t_j + h\)), we update the location of the \(i\)th particle according to
\[
    x_i(t_{j+1}) = x_i(t_j) + \frac{mh}{Zd_{i,m}} + \sqrt{2\nu h}W_i
\]
where \(Z\) and \(\nu\) are parameters of the motion, \(d_{i,m}\) represents the distance between the \(m\)th particle to the left and right of our \(i\)th particle, and \(W_i\sim\cN(0,1)\) are i.i.d standard normal random variables.
We also enforce periodic boundary conditions on \([0,2\pi)\); i.e., the particles are understood to move on the circle, and the computations of \(d_{i,m}\) take this into account.
For more details about the model and its link to Burgers' equation, we refer the reader to \cite{AK21,LKGK03}, but we note that do not ever use the Burgers' equation PDE in computing any of our approximations --- we only present it to demonstrate that our method works on particle systems that model PDEs.

\begingroup
\renewcommand{\arraystretch}{1.25}
\begin{table}[tbp]
\begin{center}
\begin{tabular}{|l|l|}
    \hline
    Parameter & Description \\
    \hline
    \(N = 20{,}000\) & number of independent particles \\
    \(h = 1/4096\) & micro-scale simulation step size \\
    \(x_i(0) \sim \cN(\pi,\tfrac12)\) i.i.d.\ \(\quad \forall i=1,\ldots,N\) & initial particle locations \\
    \(T = 2\) & end time \\
    \(Z = 1000\) & particle coupling strength \\
    \(\nu = 2\) & kinematic viscosity \\
    \(m = 200\) & number of particles used in \(d_{i,m}\) \\
    \hline
\end{tabular}
\caption{Parameters for Burgers' model micro-scale simulation.}
\label{tab:burgers-parameters-sim}
\end{center}
\end{table}
\endgroup

\begingroup
\renewcommand{\arraystretch}{1.25}
\begin{table}[ht]
\begin{center}
\begin{tabular}{|l|l|}
    \hline
    Parameter & Description \\
    \hline
    \(S = 1/8\) & length of start-up time \\
    \(k = 48\) & number of centered differences to average \\
    \(H = 37/256\) & length of each Euler-type prediction step \\
    \(H_R = 3/32\) & burn-in time after each Euler step \\
    \(N_T = 7\) & number of Euler steps \\
    \(R = 1/8\) & length of final recovery time after last Euler step \\
    \hline
\end{tabular}
\caption{Additional parameters for approximation algorithm with Burgers' model micro-scale simulation.}
\label{tab:burgers-parameters-alg}
\end{center}
\end{table}
\endgroup

\begin{figure}[tbp]
    \centering
    \includegraphics[width=0.9\textwidth]{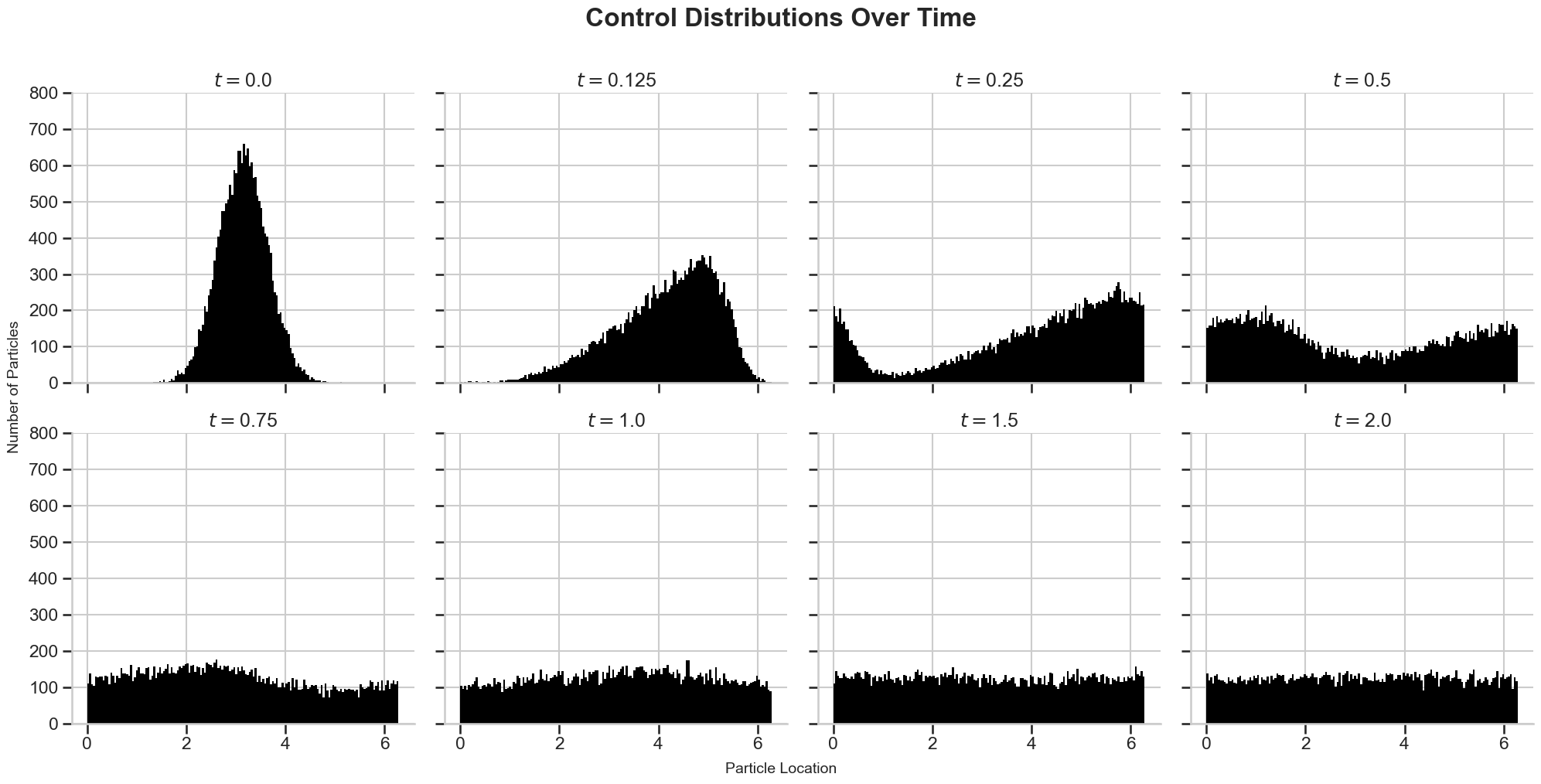}
    \caption{Histograms of particle locations at various times \(t\) from \(t=0\) to \(t=2\) for the control simulation of the Burgers' model.}
    \label{fig:burgers-control}
\end{figure}

As before, we first run the micro-scale model from \(t=0\) to \(t=2\) with the parameters in \ref{tab:burgers-parameters-sim}.
The initial particle locations \(x_i(0)\) are chosen randomly by independently sampling from the Gaussian distribution \(\cN(\pi,\tfrac12)\).
Histograms of particle locations at different times in the control simulation are shown in Figure \ref{fig:burgers-control}.

\begin{figure}[tbp]
    \centering
    \includegraphics[width=0.9\textwidth]{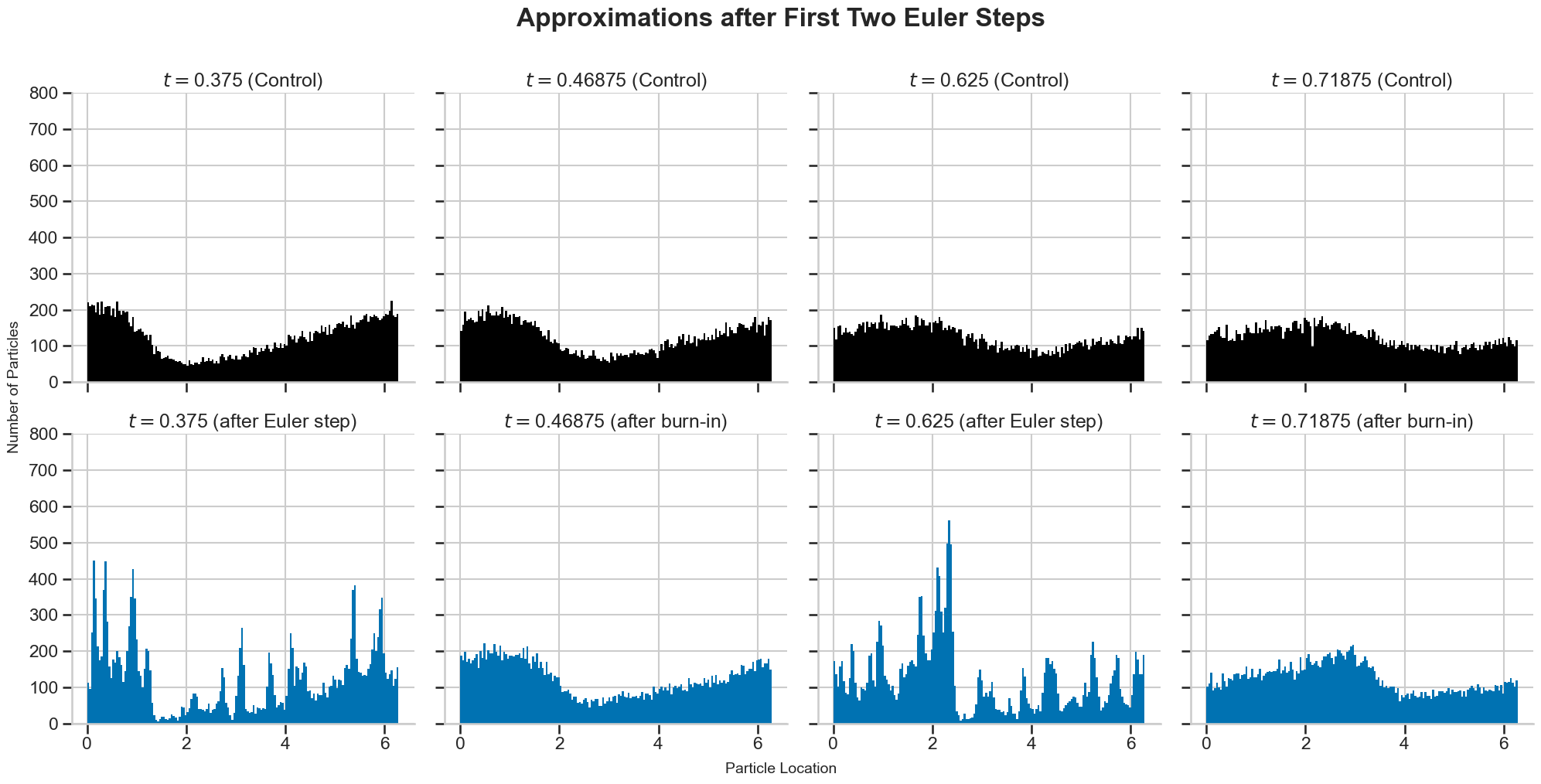}
    \caption{
    Histograms of particle locations after first two Euler steps (bottom) compared with corresponding times from control simulation (top).
    We run the micro-scale simulation until \(t=11/32\) (not pictured), take an Euler step to \(t=3/8\) (leftmost), burn in until \(t=15/32\) (second from left), take an Euler step to \(t=5/8\) (second from right), and burn in until \(t=23/32\) (rightmost).
    }
    \label{fig:burgers-approx-2i}
\end{figure}
\begin{figure}[tbp]
    \centering
    \includegraphics[width=0.9\textwidth]{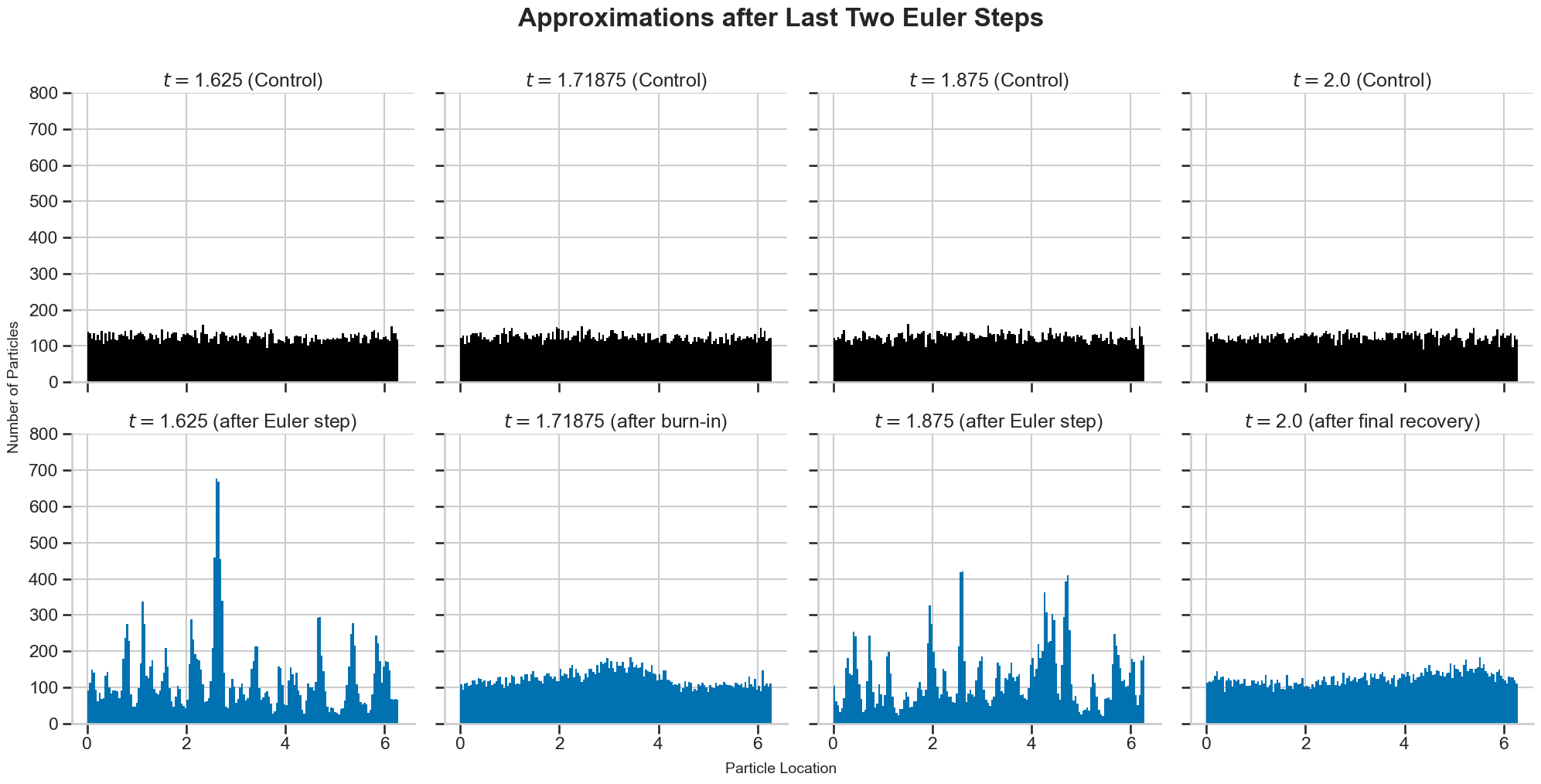}
    \caption{Histograms of particle locations after final two Euler steps (bottom) compared with corresponding times from control simulation (top). We take an Euler step from \(t=47/32\) (not pictured) to \(t=13/8\) (leftmost), burn in until \(t=55/32\) (second from left), take an Euler step to \(t=15/8\) (second from right), and finally burn in until \(t=2\) (right-most), the end time.}
    \label{fig:burgers-approx-2ii}
\end{figure}

We then run the approximation algorithm with the same parameters as in Table \ref{tab:burgers-parameters-sim} and the additional parameters (as defined in Section \ref{subsec:discrete-algorithm} in Table \ref{tab:burgers-parameters-alg}.
Note that the effective end time of this simulation is \(T = S+N_T(H_S + H + H_R) + R = 2\) (since \(H_S + H + H_R = 1/4\)), which agrees with the control simulation.
Figures \ref{fig:burgers-approx-2i} and \ref{fig:burgers-approx-2ii} show the distributions obtained after the first two and last two Euler steps (and corresponding burn-in periods) respectively.
We observe that, after the burn-in period, the approximated distributions do indeed reasonably resemble the control distributions at the corresponding times.
We also notice that the distributions immediately after the Euler step are much ``spikier'' than the control distributions.
This highlights the importance of the burn-in period --- the spiky distribution quickly returns to a relatively smooth distribution that much more closely resembles the control distribution, at which point the velocity field approximation will be much closer to the velocity field describing the control distribution's evolution.
If the velocity field were approximated immediately after the Euler step, the effect of the distribution ``smoothing out'' would dominate the effect of the long-term evolution, which is the piece we wish to approximate.

\begin{rmk}
    It is worth pointing out that the spikiness observed in the approximated distributions after the Euler step is a result of the approximated velocity field \(\bv_t^{k,N,h}\) not being ``smooth enough,'' as other experimental evidence suggests that smoother velocity fields yield smoother approximated distributions.
    Thus, this spikiness could likely be reduced by, e.g., first applying a kernel smoothing technique to this velocity field, and doing so could drastically reduce the burn-in time required for the distribution to return to a relative equilibrium.
\end{rmk}

To compare the computational cost of our method compared to the control simulation, we once again compute the number of micro-scale steps needed.
Each step of our algorithm only requires \(15/32\) as many micro-scale steps as the control simulation (\(3/128\) seconds of velocity-field estimation and \(3/32\) seconds of burn-in for every \(1/4\) second of simulated time), and ignoring burn-in, the ratio drops to \(29/128\).
While not substantial, this still demonstrates that our method provides nontrivial improvement in terms of the number of micro-scale steps needed to accurately take a macro-scale time step.

We finally note that the periodic boundary conditions make it more difficult to directly apply the JKOnet* approach to learn a potential.
Thus, we do not provide a comparison to using this method as we do not believe it fairly reflects the setting where JKOnet* performs well.
We do point out, however, that another advantage of our algorithm is its ability to easily generalize to particles evolving in spaces which are not \(\R^d\), as optimal transport maps between discrete distributions on manifolds are equally well defined.
For another model for which we can fairly compare our algorithm to the equivalent JKOnet* approach, see the following section.

\begin{rmk}
    As a final note, we point out that one can learn a PDE that models all three of our simulations (see \cite{LPSK23} for a discussion of PDEs and the chemotactic model from Section \ref{sec:chemotaxis}, and note that the Langevin dynamics model in Section \ref{sec:halfmoon} can be understood using the Fokker-Planck equation).
    This indicates that our method works for particle models of PDEs, though the practical computational improvement might not be substantial because standard numerical methods for (known) PDEs are relatively cheap.
    In contrast, our method is better suited for the setting where the PDE is not known, and instead one is only presented with a particle model.
\end{rmk}

\section{Simulation Study III: Langevin Dynamics} \label{sec:halfmoon}

One of the key and motivating advantages of framing our algorithm in terms of optimal transport is that it easily generalizes to multi-dimensional models.
We perform one final experiment as a proof of concept demonstrating that our algorithm works with a 2D model as well.
For simplicity, the micro-scale model we use is an Euler--Maruyama simulation of an SDE: at each micro-scale time step of size \(h\), each particle's position \(x_i(t_j)\) is updated according to
\[
    x_i(t_{j+1}) = x_i(t_j) - \nabla U(x_i(t_j))\cdot h + \sqrt{2h}\cdot Z_i(t_j)
\]
where \(Z_i(t_j)\sim \cN(0,I_2)\).
This is the standard Euler--Maruyama simulation of the SDE
\[
    dX_t = -\nabla U(X_t)dt + \sqrt2 dW_t
\]
with potential \(U(x)\).
The potential we use for our micro-scale model is a half-moon potential given by
\[
     U(x,y) = A \left(r(x,y) - R\right)^2 + B \exp\left( - \alpha (y - y_s)\right)
\]
with \(A = 2, B = 0.5, R = 2, \alpha = 1.5, y_s = -0.5\), and where the function \(r(x,y) = \sqrt{x^2+y^2}\) measures the distance from the origin.

\begin{rmk}
    We reiterate that this micro-scale model is computationally cheap over long time scales, so in practice, there is little computational benefit to using our algorithm in this case.
    We present it as a demonstration of the algorithms performance on a 2D model, and in particular one whose steady-state distribution is not as simple as a standard Gaussian or uniform distribution.
\end{rmk}

\begingroup
\renewcommand{\arraystretch}{1.25}
\begin{table}[tbp]
\begin{center}
\begin{tabular}{|l|l|}
    \hline
    Parameter & Description \\
    \hline
    \(N = 10{,}000\) & number of independent particles \\
    \(h = 1/2048\) & micro-scale simulation step size \\
    \(x_i(0) \sim \operatorname{Unif}([-4,4]\times[-4,4])\) i.i.d.\ \(\quad \forall i=1,\ldots,N\) & initial particle locations \\
    \(T = 2\) & end time \\
    \hline
\end{tabular}
\caption{Parameters for 2-D half-moon micro-scale simulation.}
\label{tab:halfmoon-parameters-sim}
\end{center}
\end{table}
\endgroup

\begingroup
\renewcommand{\arraystretch}{1.25}
\begin{table}[ht]
\begin{center}
\begin{tabular}{|l|l|}
    \hline
    Parameter & Description \\
    \hline
    \(S = 1/8\) & length of start-up time \\
    \(k = 32\) & number of centered differences to average \\
    \(H = 7/32\) & length of each Euler-type prediction step \\
    \(H_R = 1/64\) & burn-in time after each Euler step \\
    \(N_T = 7\) & number of Euler steps \\
    \(R = 1/8\) & length of final recovery time after last Euler step \\
    \hline
\end{tabular}
\caption{Additional parameters for approximation algorithm with chemotaxis micro-scale simulation.}
\label{tab:halfmoon-parameters-alg}
\end{center}
\end{table}
\endgroup

\begin{figure}[tbp]
    \centering
    \includegraphics[width=0.9\textwidth]{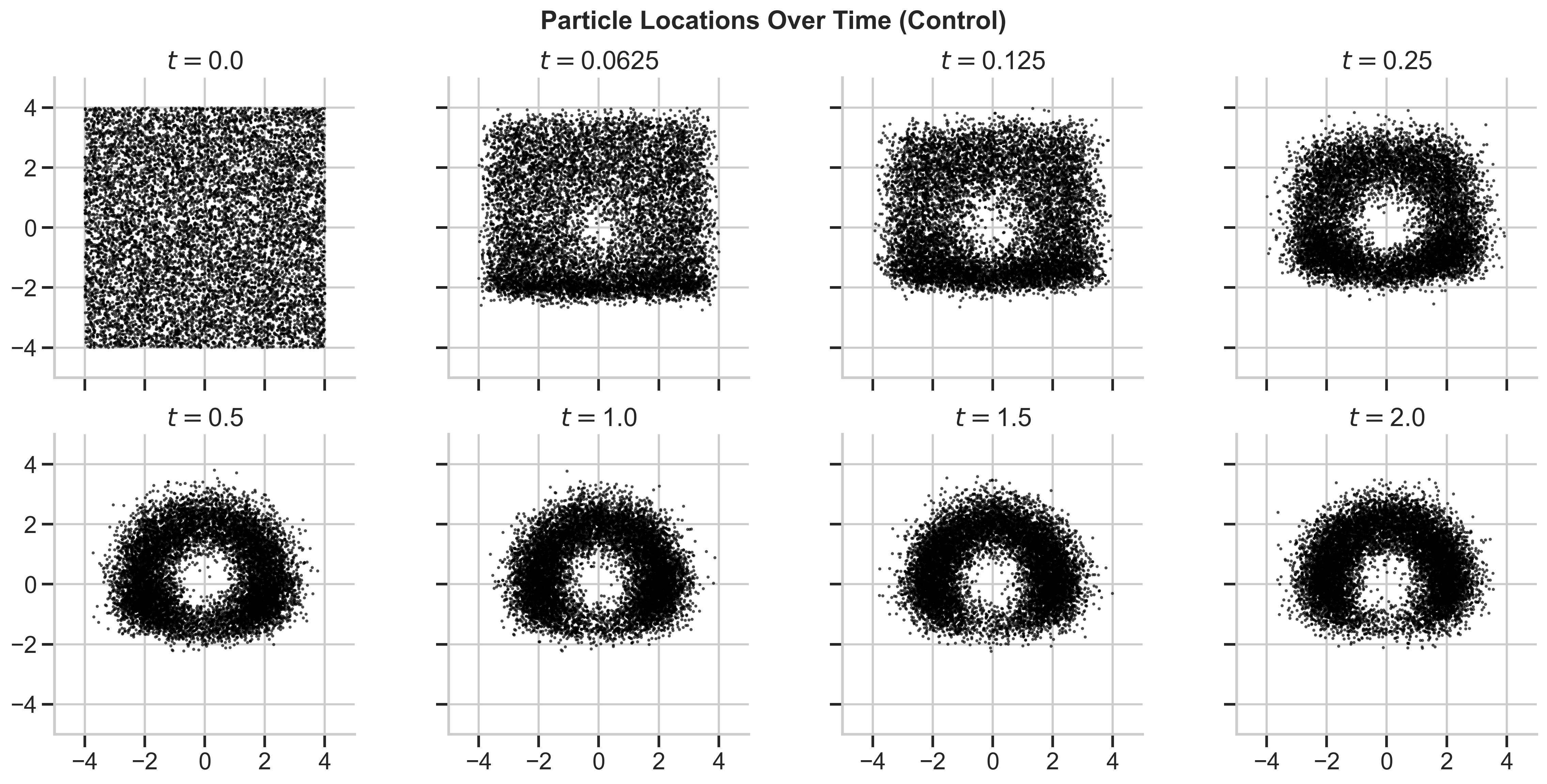}
    \caption{Scatter plots of particle locations at various times \(t\) from \(t=0\) to \(t=2\) for the control simulation of the half-moon Langevin dynamics model.}
    \label{fig:halfmoon-control}
\end{figure}

\begin{figure}[tbp]
    \centering
    \includegraphics[width=0.9\textwidth]{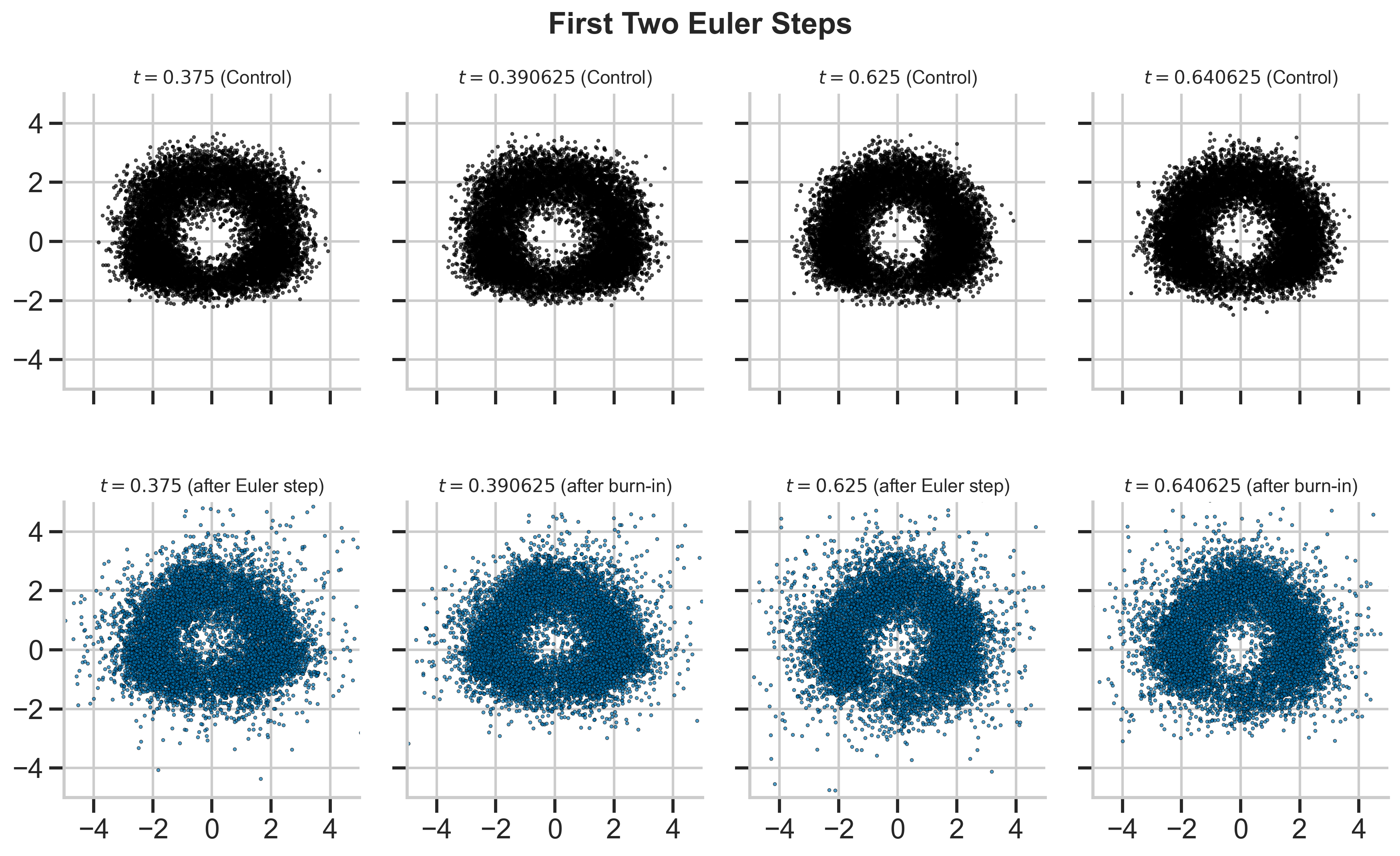}
    \caption{
    Scatter plots of particle locations after first two Euler steps (bottom) compared with corresponding times from control simulation (top).
    We run the micro-scale simulation until \(t=9/64\) (not pictured), take an Euler step to \(t=3/8\) (leftmost), burn in until \(t=25/64\) (second from left), take an Euler step to \(t=5/8\) (second from right), and burn in until \(t=41/64\) (rightmost).
    }
    \label{fig:halfmoon-approx-2i}
\end{figure}
\begin{figure}[tbp]
    \centering
    \includegraphics[width=0.9\textwidth]{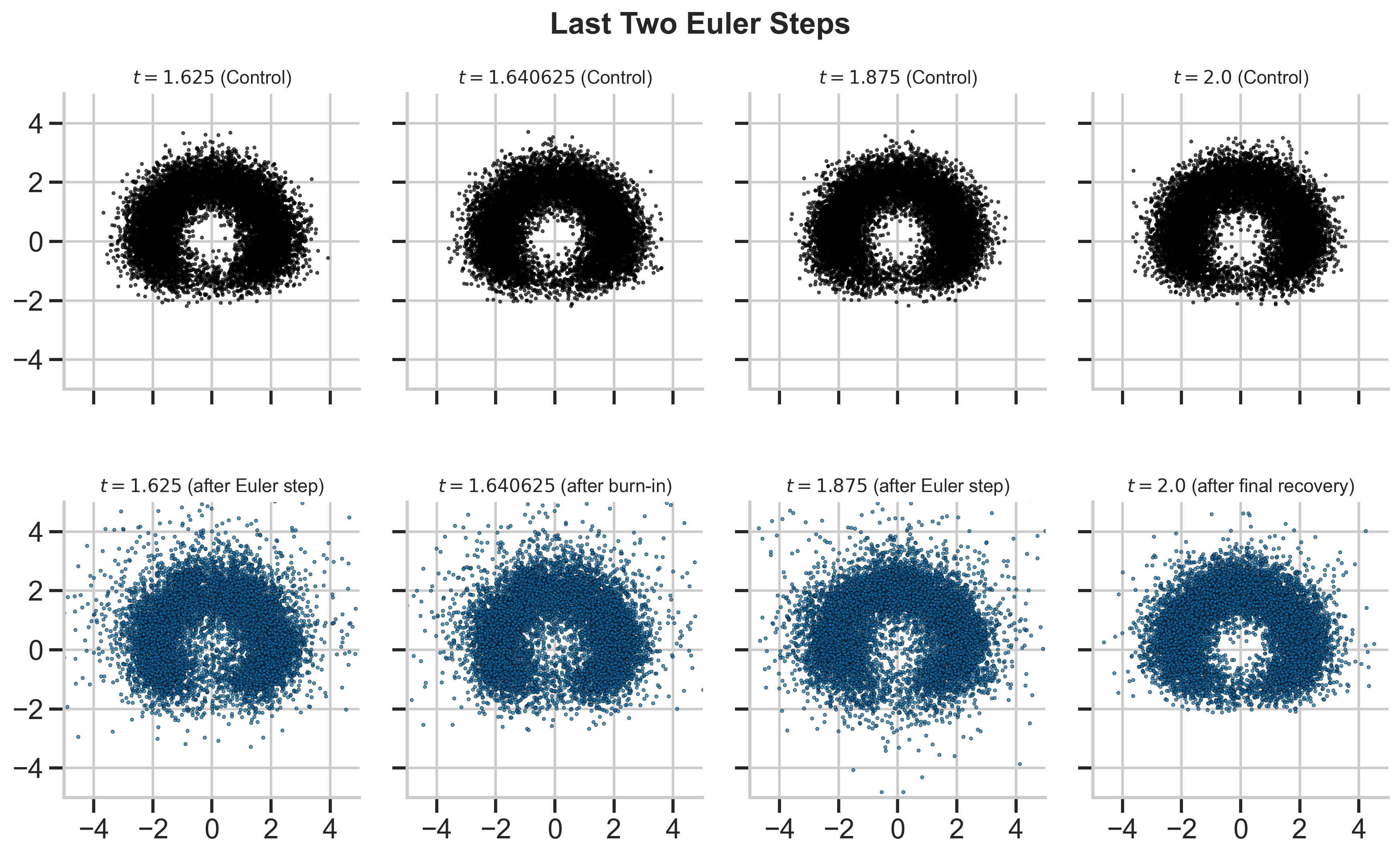}
    \caption{Scatter plots of particle locations after final two Euler steps (bottom) compared with corresponding times from control simulation (top). We take an Euler step from \(t=89/64\) (not pictured) to \(t=13/8\) (leftmost), burn in until \(t=105/32\) (second from left), take an Euler step to \(t=15/8\) (second from right), and finally burn in until \(t=2\) (right-most), the end time.}
    \label{fig:halfmoon-approx-2ii}
\end{figure}

As before, we run the micro-scale simulation from \(t=0\) to \(t=T\) with the parameters shown in Table \ref{tab:halfmoon-parameters-sim} (the end time is once again chosen long enough to effectively reach steady state).
Figure \ref{fig:halfmoon-control} shows scatter plots of the particles over time for this ``control'' simulation.

We then run the approximation algorithm with the same parameters as in Table \ref{tab:halfmoon-parameters-sim} and the additional parameters (as defined in Section \ref{subsec:discrete-algorithm}) in Table \ref{tab:halfmoon-parameters-alg}.
Note that the effective end time of this simulation is \(T = S + N_T(H_S + H + H_R) + R = 2\) (since \(H_S + H + H_R = 1/4\)), which agrees with the control simulation.
Figures \ref{fig:halfmoon-approx-2i} and \ref{fig:halfmoon-approx-2ii} show the distributions obtained after the first two and last two Euler steps, respectively.
Aside from a few particles straying far from the bulk distribution, we see that the approximated distributions closely match the shape and density of the control distributions, and the approximations only require \(3/16\) as many micro-scale steps as the control simulation (\(1/32\) seconds of velocity-field estimation and \(1/64\) seconds of burn-in for every \(1/4\) of simulated time).

\begin{figure}[tp]
    \begin{subfigure}[t]{0.49\textwidth}
        \center{\includegraphics[width=\textwidth,height=2.25in]{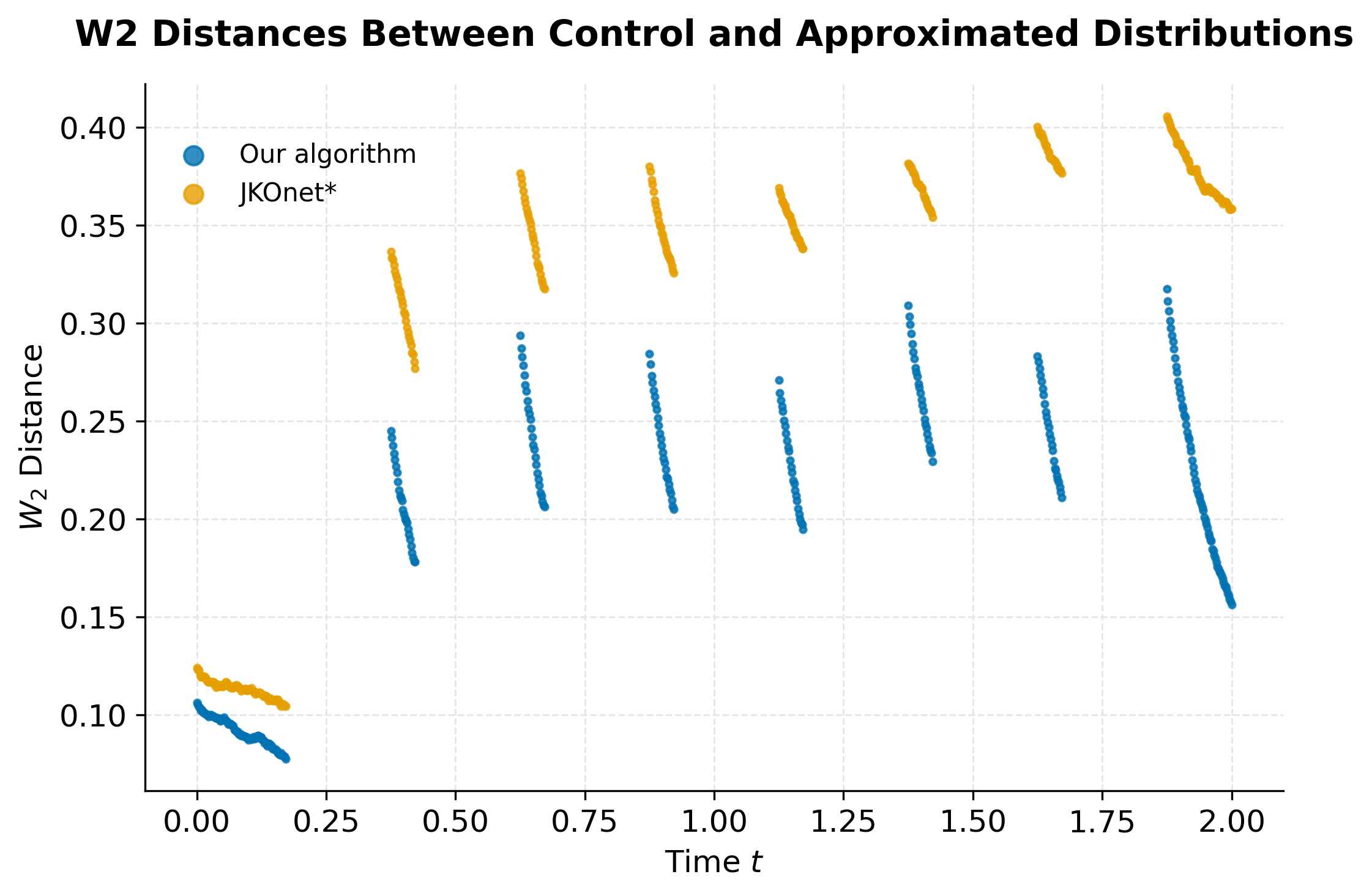}}
        \caption{2-Wasserstein distance between control distribution and approximated distributions versus time.}
        \label{fig:halfmoon-compare-schemes-1}
    \end{subfigure}
    \hfill
    \begin{subfigure}[t]{0.49\textwidth}
        \center{\includegraphics[width=\textwidth]{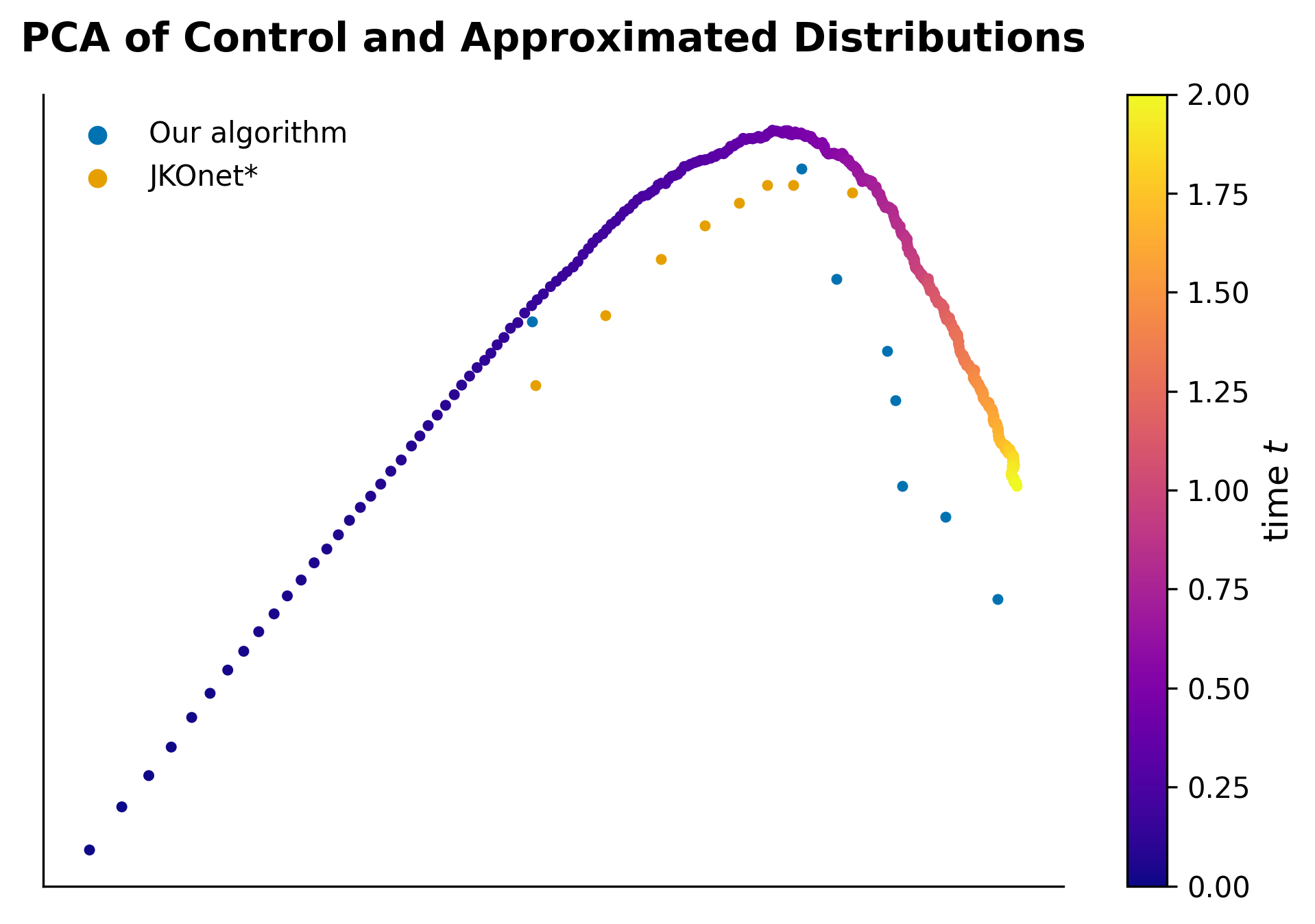}}
        \caption{Projections of distributions onto the two principal components of the control distribution's evolution over time.
        The solid rainbow curve shows the control distribution, where the color represents time according to the color bar at the right.
        }
        \label{fig:halfmoon-compare-schemes-0}
    \end{subfigure}
    \caption{Comparisons between control distributions and approximations.
    The blue is our algorithm, and the orange is using the JKOnet* method for learning and simulating an SDE.}
    \label{fig:halfmoon-compare-approximations}
\end{figure}

We also perform the equivalent algorithm using JKOnet* \cite{TLGD24} to use the \(2k+1 = 65\) micro-scale steps to learn a potential and then simulate an SDE with appropriate potential for the equivalent duration of the Euler step.
We compare the \(W_2\) distance between the control distribution and approximated distributions for both approximation methods in Figure \ref{fig:halfmoon-compare-schemes-1}, and in Figure \ref{fig:halfmoon-compare-schemes-0}, we perform PCA (described in detail in Section \ref{subsec:chemotaxis-experimental-results}) to visualize how well the approximated distributions track the trajectory of the control distribution's evolution.
In both figures, we see that our method outperforms the analogous JKOnet* approach, which suggests that our approach is better able to learn the local evolution.
In particular, it appears that the JKOnet* approximated distributions do not update nearly enough over each step, which indicates that the learned potential is too mild.
This is consistent with the understanding that the ``local'' data over short micro-scale bursts is not sufficient to learn a global potential, and it suggests that our approach of approximating a velocity field with optimal transport more efficiently learns the local evolution in order to predict the distribution some time in the future.

\section*{Acknowledgements}
\addcontentsline{toc}{section}{Acknowledgments}
{\rm E.\ A.\ Nikitopoulos acknowledges support from NSF grant DGE 2038238.
The work of I.\ G.\ Kevrekidis was partially supported by the US Department of Energy.
A.\ Cloninger acknowledges support from NSF grants DMS 2012266 and CISE CCF 2403452.
}

\bibliographystyle{plain}
\bibliography{BibliographyUnified}

\end{document}